\documentclass[11pt,english]{article}
\usepackage[T1]{fontenc}
\usepackage[latin9]{inputenc}
\usepackage{geometry}
\usepackage{color}
\usepackage{amsthm, amsmath}
\usepackage{amssymb}
\usepackage{graphicx,url}
\usepackage{setspace}
\usepackage{times}
\usepackage[authoryear]{natbib}
\makeatletter

\@ifundefined{date}{}{\date{}}
\newcommand*{\nindep}{%
	\mathbin{%                   % The final symbol is a binary math operator
		\mathpalette{\@indep}{\not}% \mathpalette helps for the adaptation
	}%
}
\newcommand*{\@indep}[2]{%
	\sbox0{$#1\perp\m@th$}%        box 0 contains \perp symbol
	\sbox2{$#1=$}%                 box 2 for the height of =
	\sbox4{$#1\vcenter{}$}%        box 4 for the height of the math axis
	\rlap{\copy0}%                 first \perp
	\dimen@=\dimexpr\ht2-\ht4-.2pt\relax
	\kern\dimen@
	{#2}%
	\kern\dimen@
	\copy0 %                       second \perp
}

\newcommand{\bea}{\begin{eqnarray*}}
\newcommand{\eea}{\end{eqnarray*}}
\newcommand{\be}{\begin{eqnarray}}
\newcommand{\ee}{\end{eqnarray}}
\newcommand{\bay}{\begin{array}}
\newcommand{\eay}{\end{array}}
\newcommand{\ben}{\begin{enumerate}}
\newcommand{\een}{\end{enumerate}}
\newcommand{\bcen}{\begin{center}}
\newcommand{\ecen}{\end{center}}

\def\P{\mathrm{pr}}
\def\E{ E}
\def\V{\mathrm{var}}

\def\Sp{\mathrm{span}}
\def\Tr{\mathrm{tr}}

\def\ind{ 1}
\def\R{ \mathbb{R} }
\def\Rp{ \mathbb{R}^p }
\def\Rd{ \mathbb{R}^d }

\def\tU{ \hat{U} }
\def\AIC{\textsc{aic}}

\newcommand{\T}{\mathrm{\scriptscriptstyle T}}

\DeclareOldFontCommand{\bf}{\normalfont\bfseries}{\mathbf}

\newtheorem{thm}{Theorem}[section]
\newtheorem{lem}{Lemma}[section]
\newtheorem{prop}{Proposition}[section]

\newtheorem{remark}{Remark}[section]

\usepackage{babel}

\makeatother

\usepackage{subfigure}

\begin{document}

	\title{\textbf{\huge{}Nonparametric principal subspace regression}}
	 \author{Mark Koudstaal \\
  Department of Statistical Sciences, University of Toronto \\[10pt]
  Dengdeng Yu\\
  Department of Statistical Sciences, University of Toronto \\[10pt]
  Dehan Kong\\
  Department of Statistical Sciences, University of Toronto \\[10pt]
  Fang Yao \\
  Department of Probability \& Statistics, Peking University \\[10pt]
    }
	\maketitle

\maketitle

\begin{abstract}
In scientific applications, multivariate observations often come in tandem with temporal or spatial covariates, with which the underlying signals vary smoothly.  The standard approaches
such as principal component analysis and factor analysis neglect the smoothness of the data, while multivariate linear or nonparametric regression fail to leverage the correlation information among multivariate response variables. We propose a novel approach named nonparametric principal subspace regression to overcome these issues. By decoupling the model discrepancy, a simple and general two-step framework is introduced, which leaves much flexibility in choice of model fitting. We establish theoretical property of the general framework, and offer implementation procedures that fulfill requirements and enjoy the theoretical guarantee. We demonstrate the favorable finite-sample performance of the proposed method through simulations and a real data application from an electroencephalogram study. 
\end{abstract}

{\bf Keywords:}
Factor model, nonparametric principal subspace, singular value decomposition, smoothness.

\newpage

\section{Introduction}
\label{Introduction}
In scientific applications, one is often interested in predicting a multivariate response using one or a few predictor variables. The multivariate response linear regression is a conventional way to model this type of data. The usual procedure is the ordinary least squares, equivalent to performing individual linear regression of each response variable on predictor variables, which fails to utilize the correlation information among the response variables. To incorporate the correlation information, \cite{breiman1997predicting} proposed a multivariate shrinkage method to leverage information from the correlation structure, which helps to improve the predictive accuracy compared to the ordinary least squares. 

Although multivariate response linear regression is a useful tool, it may not work properly in some applications. For example, in the electroencephalogram application presented in Section \ref{EEGdata},  we are interested in modelling the dynamic changes of the electroencephalogram signals detected from 64 electrodes of the scalp. For each electrode, the signal is sampled at 256 Hz per second. We plot the electroencephalogram signals from one randomly selected participant in Figure \ref{eegPlots}(a), where the curves show nonlinear patterns. This indicates that the multivariate response linear regression model may not be adequate to characterize the relationship between the common predictor time and the multivariate signals from the 64 electrodes. A natural rescue is to utilize nonparametric regression of the multivariate response variables on the common predictors.  However this solution is unsatisfactory as performing individual nonparametric regressions does not capture correlations among the response variables.  %{\color{red} The sentence above was edited from: However, the equivalence to performing individual nonparametric regression again fails to capture the correlation among the response variables.} {\color{blue} ``the equivalence to performing" ...  reads unclearly to me.  Have put a suggested edit in to capture what I thought you might have meant. }

Motivated by the application, we propose a new nonparametric principal subspace regression model, which allows more flexible nonlinear structures of the regression functions while takes into account the correlation among response variables of the same time.  Our proposal is related to the factor models, which characterize the correlation structure in multivariate data. %and has found successful applications ranging from economics and epidemiology \cite{bai2002determining, stock2002forecasting, durante2014locally, fox2015bayesian} to geophysical and biomedical signal processing \cite{maris2003resampling, chang2004estimation}. 
In factor models, the signal of interest is expressed as a linear combination of a few latent variables, and does not concern additional covariate information that may play a role in estimation or prediction. For instance, factor models are often employed in contexts such as multiple time series or correlated functional data \citep{engle1981one, huang2009analysis}, where useful information may be hidden in the form of smoothness with respect to some additional covariates, e.g. temporal or spatial variable. Neglecting such information in recovery and prediction potentially hinders the quality and performance of resulting estimators. This has been noticed by \cite{durante2014locally}, which further proposed a locally adaptive factor process under the Bayesian framework for characterizing multivariate mean-covariance changes in continuous time, allowing locally varying smoothness in both the mean and covariance matrix of multivariate time series. However, theoretical guarantees are lacking for the approach, which may leave practitioners uncertain about the quality of resulting estimates. 

In this work, we approach the problem from a different perspective that is intuitive and broadly applicable. The contributions are summarized as follows. First, we propose a new nonparametric principal subspace regression model. This not only incorporates the correlation structure among multivariate responses, but also accounts for nonlinear trend and smoothness of the data. Second, we introduce a simple two-step estimation framework, where the first step is to obtain the orthogonal left singular vectors and the second step is to estimate the nonparametric loading functions. This procedure is general and leaves flexibility in choice of model fitting. Third, we provide theoretical guarantees for the general proposal, and then present some examples of standard linear smoothers and the rates they attain when used in the general proposal. Lastly, we show that our method outperforms its counterpart, the conventional nonparametric regression, in simulations as well as an electroencephalogram study. This is not surprising because our approach significantly reduces the model complexity and risk of overfitting compared to individual nonparametric regressions. 
%{\color{blue} 
%I think that it might be valuable to expand the scope of outperform here as there are some other advantages to approaching recovery from our modelling perspective.  Some of these include:
%\begin{itemize}
%\item Model parsimony: significant reduction of model size from individual nonparametric regressions. 
%\item Reduced risk of overfitting. 
%\item Less parameters to tune, and less regressions to perform leads to significantly reduced computational burden.
%\end{itemize}
%}

The rest of the paper is organized as follows. In Section \ref{ProposedMethodology}, we propose the nonparametric principal subspace regression methodology, state the main
results and three important examples. We further give a specific fitting procedure for our approach. Section \ref{numerical} evaluates the favorable finite-sample performance of the proposed method through simulations and a real data application from an electroencephalogram study.  Proofs of main propositions and theorems as well as the statements for relevant lemmas are contained in
Sections \ref{Appendixprooftheorem} and \ref{lemmasMainTheorem}. Statements of auxiliary lemmas and theorems for examples are given in Section \ref{lemmasMainExamples}, whereas some additional simulation results and proofs of all lemmas and theorems for examples are deferred to the supplementary file.

\section{Proposed Methodology and Theoretical Guarantees}
\label{ProposedMethodology}
\subsection{Notation}
\label{Setup:Notation}
Denote the inner product of $a,b \in \R^m$ by $ \langle a , b \rangle = a^\T b=\sum_{i=1}^m a_i b_i$, where $ a_i $ and $ b_i $ are the $i$th components of $ a $ and $ b $ respectively. Let $\| \cdot \|$ be the corresponding norm induced by the inner product. Define the rescaled inner
product $   \langle a , b \rangle_m  =  \frac{1}{m}  \langle a , b \rangle  $ and the induced norm $\| \cdot \|_m$. For two functions $f,g \in L^2$, the inner product and corresponding norm bear the subscript
$L^2$, i.e. $\langle f , g \rangle_{L^2}=\int_T f(x)g(x)dx $ and $ \| f\|^2_{L_2}=\int_T \{f(x)\}^2dx $,  where $ T $ is the domain of $ x $. Let $ \|a\|_{\infty}=\max_{1\leq i\leq m} |a_i|$ and $ \|f\|_{\infty} $ denote the sup norm of function $f \in L^2$.  Suppose we have a matrix $M\in  \mathbb{R}^{p\times n}$ with rank $r$, and its singular values satisfy $\sigma_1(M) \ge \sigma_2(M) \ge \ldots \ge \sigma_r(M)>\sigma_{r+1}(M)=\ldots=0$. Consider the singular value decomposition $M = U \Sigma V^\T  $, where $U$ and $V$ are $p\times r$ and $n\times r$ matrices respectively with orthonormal columns, and $ \Sigma = \textrm{diag}(\{\sigma_j(M)\}_{1\leq j\leq r})$. The spectral norm of $M$,
denoted by $\|M \|$, is defined by $\| M \| = \sigma_1(M)$ and the Frobenius norm $\| M \|_F$ is defined as $\| M \|^2_F =  \sum_{1\leq j\leq r} \sigma_j^2(M) = \sum_{1\leq k\leq p, 1\leq i\leq n} M_{ki}^2$. Suppose we have two $p \times q$ orthonormal collections $U$ and $V$ with $ p\geq q $, and $\sigma_1 \ge \sigma_2 \ge \dots \ge \sigma_q \ge 0$ are singular values of $U^\T V$. Define the principal angles between two matrices $U$ and $ V $ as $\Theta(U,V) = \textrm{diag} ( \cos^{-1} (\sigma_1) , \dots, \cos^{-1} (\sigma_q)  )$. Applying sinusoid elementwise and taking the spectral
norm gives the $\sin \Theta$ distance between $U$ and $V$, denoted by $\| \sin
\Theta(U,V) \|$.

\subsection{Nonparametric principal subspace regression}
\label{Setup:Data}
Let $\{(x_i , y_i), i=1,\ldots,n\} $ be independent and identically distributed observations, where $ x_i\in [0, 1]^d$ and $ y_i\in \Rp $. We consider the following nonparametric model
\be \label{nonparametricFactorData}
y_i  = F(x_i)  + z_i, 
\ee
with $ z_i=(z_{i1}, \ldots, z_{ip})^{\T} \in \Rp $ is independent and identically distributed with mean $ 0 $ and covariance $\Sigma $, and $F : [0, 1]^d \mapsto \Rp$ so that  $\E (  y_i | x_i  )  =  F(x_i)$.  Our goal is to estimate the function $ F(\cdot) $, which characterizes the relationship between $ x_i $ and $ y_i $, under smoothness assumptions on the components of $F(\cdot)$.  %This puts us squarely in the domain of $p$ repeated nonparametric experiments where the uncorrelated assumption of Gaussian noise is commonly used to facilitate model exploration and theoretical development \citep{cai2012, donoho1994ideal, Tsybakov:2008:INE:1522486, Johnstone2017gsm}.
%{\color{red} Above sentence edited from: This puts us squarely in the domain of $p$ repeated nonparametric experiments where uncorrelated Gaussian noise assumption is a common used assumption to facilitate the theoretical development \citep{cai2012, donoho1994ideal, Tsybakov:2008:INE:1522486, Johnstone2017gsm}.}
%{\color{blue} Double use of assumption and wording reads a bit akwardly. Have put a suggested edit in. }

%Let $\{(x_i , y_i), i=1,\ldots,n\} $ be independent and identically distributed observations, where $ x_i\in \R$ and $ y_i\in \Rp$. We allow $ p $ to diverge with sample size $ n $, and consider the following nonparametric model
%\be \label{nonparametricFactorData}
%\E (  y_i | x_i  )  = F(x_i),
%\ee
%where $F : [0, 1]^d \mapsto \Rp$.   
%{\color{red} Based on the theorems and examples, probably easiest to use $[0, 1]^d$ here rather than $\R$.  Also, this definition changes the data definition from $y_i = F(x_i) + \sigma z_i$, $z_i \stackrel{i.i.d.}{\sim} N_p(0, I)$ which was part of the initial versions, is central to the theoretical development and also answers some of the questions that you had.  The theorems and supplementary material all rely on the earlier definition; are you sure you would like to change these assumptions and notation?}

Motivated by factor analysis and the singular value decomposition as methods of accounting for correlation among variables in the $ y_i $, we assume that $F(\cdot)$ lies in a low dimensional subspace of $\R^p$ and can be written as

%{\color{blue} The above sentence reads a bit awkwardly to me.  A couple of suggested edits are:

%Motivated by factor type models as a method of accounting for correlation among variables in the $ y_i $, such as factor analysis and the singular value decomposition, we assume that $F(\cdot)$ lies in a low dimensional subspace of $\R^p$ and can be written as

%or

%Motivated by factor analysis and the singular value decomposition as methods of accounting for correlation among variables in the $ y_i $, we assume that $F(\cdot)$ lies in a low dimensional subspace of $\R^p$ and can be written as
%}
\be  \label{nonparametricFactorizationModel}
     F(x) =  \sum_{k = 1}^q  f_k(x) u_k,
\ee
where $ q\leq p $, $u_k \in \mathbb{R}^p$ ($1\leq k\leq q$) are orthonormal vectors and $ \{f_k(\cdot), 1\leq k\leq p\} $ are smooth %{\color{red} orthogonal (The orthogonality assumption on the $f_k$ could also potentially be moved to Theorem \ref{mainConsistencyfirststep})} {\color{blue} (Mark: can you explain it in detail what does smooth orthogonal mean? I am actually confused about this. Do you mean the inner product between $ f_k $s are zero? Do we need this assumption? It looks to me that the estimates $\hat f_j$ from our procedure may not enjoy the property?)} 
functions which are orthogonal in $L^2[0, 1]^d$. 
%{\color{red} (As we are specifying that they are in $L^2[0, 1]^d$, it seems redundant to note that the $f_k$ vary with $x$.  Also, as before, we have deferred discussion of smoothness to the examples following theoretical development)}. Combining \eqref{nonparametricFactorData} and \eqref{nonparametricFactorizationModel}, one can write 
%\be \label{model}
%\E (  y_i | x_i  )  = \sum_{k = 1}^q  f_k(x_i) u_j, 
%\ee
As we show in Proposition \ref{svdRepresentation}, if one takes $q=p$, every function $F : [0, 1]^d \mapsto \Rp$ with components in $L^2[0, 1]^d$ has such a representation; hence the model has the simple interpretation of reducing dimension with smoothness dependence on covariates.   
%which not only accounts for correlation via factor analysis but also nonparametrically captures smoothness information of the covariates. We refer our model as ``nonparametric principal subspace regression''. To appreciate this framework, one special case of \eqref{nonparametricFactorizationModel} is given by the example. 
Thus, in addition to capturing correlations via factor type analysis, this model also nonparametrically incorporates smoothness information into the covariates. We refer to our model as ``nonparametric principal subspace regression''. %{\color{red} Above sentence read:  Further, this model accounts for correlation via factor type analysis while also nonparametrically captures smoothness information of the covariates. We refer our model as ``nonparametric principal subspace regression''.}

With the model in place, we aim to find an estimator $G(\cdot)$ of $F(\cdot)$ so that the sample discrepancy
\[  R_n(G) \equiv  \frac{1}{n} \sum_{i=1}^n || F(x_i) - G(x_i) ||^2   \]
is small at a proper nonparametric rate. As the regression functions and
estimates are assumed to be smooth, under the assumption that the $x_i$'s grow dense in $[0,1]^d$, $R_n(G)$ is a reasonable approximation to
\[  R(G) \equiv  \int_{[0,1]^d} || F(x) - G(x) ||^2 dx .   \]
%We will suggest a computationally efficient method for finding the
%estimates $G$, note the salient features of the method and provide a general template
%for theoretical analysis of other methods which can be used.\\
Since we do not observe $F(\cdot)$ directly, a natural surrogate for $R_n(G)$ is the empirical discrepancy,
%We aim to find an estimator $ G $ of $ F$ so that  the empirical discrepancy, 
$R^{\mathcal{D}}_n(G)  \equiv  \frac{1}{n} \sum_{i=1}^n \| y_i - G(x_i) \|^2 $, which we aim to minimize in place of $R_n(G)$.  
%{\color{red} Above edited from: Since we do not observe $F(\cdot)$, a natural surrogate for the sample discrepancy
%$R_n(G)$ is the empirical discrepancy,
%We aim to find an estimator $ G $ of $ F$ so that  the empirical discrepancy, 
%$R^{\mathcal{D}}_n(G)  \equiv  \frac{1}{n} \sum_{i=1}^n \| y_i - G(x_i) \|^2 $, which we aim to minimize as a surrogate for $R_n(G)$.}
%{\color{blue} Too many surrogates and discrepancies used in short span.}

%{\color{blue} As we are just exploring the form of our objective here over a specified class of functions, I think it might be notationally clearer to remove the hat's and use, e.g. $V$ in place of the $\hat{U}$ here}
Given the model \eqref{nonparametricFactorizationModel}, assume $G$ takes the form 
\[   G(x) = \sum_{k=1}^q g_k(x) v_k,    \]
where $V = (v_1 \ldots v_q)\in \mathbb{R}^{p\times q}$ is a set of orthonormal vectors and $ \{g_k(\cdot), 1\leq k\leq p\} $ are smooth functions.  Then we form the projection matrix $P_{V} = V V^{\T}$ and note that $I_p - P_{V}$ and $P_{V}$ project onto orthogonal subspaces.  Thus for  any   $x, y \in \R^p$ we have $(I_p - P_{V}) x \perp P_{V} y$ as $\langle (I_p - P_{V}) x , P_{V} y \rangle  = x^{\T} (I_p - P_{V}) P_{V} y = x^{\T} (P_{V} - P_{V}) y = 0$, since $P_{V} P_{V} =  V V^{\T}  V V^{\T} = V I_q V^{\T} = V V^{\T}   = P_{V}$. Then for a given $i$, we write
\begin{eqnarray}
 || y_i  -  V g (x_i) ||^2 &=& || (I_p - P_{V}  )y_i  + P_{V} \{y_i  -  V g (x_i) \} ||^2    \nonumber \\
 &=& || (I_p - P_{V}  )y_i  ||^2  + || P_{V} \{ y_i  -  V g (x_i) \} ||^2.    \nonumber
\end{eqnarray}
Setting $c(Y,V ) \equiv \frac{1}{n}\sum_{i=1}^n || (I_p - P_{V}  )y_i  ||^2 $, which is independent of $G$, and observing that  
\[ || P_{V} ( y_i  -  V g (x_i) ) ||^2  = || V^{\T} y_i  -   g (x_i) ) ||^2  = \sum_{k=1}^q \{ v_k^{\T} y_i  -   g_k(x_i) )\}^2, \]
we may decompose the objective $R^{\mathcal{D}}_n(G)$ as
\[  R^{\mathcal{D}}_n(G) = c(Y,V ) + \sum_{k=1}^q  R^{\mathcal{D}}_n(v_k, g_k) \hspace{5pt} \text{ where }
\hspace{5pt} R^{\mathcal{D}}_n(v_k, g_k)  \equiv  \frac{1}{n}   \sum_{i=1}^n  \{v_k^{\T} y_i  - g_k(x_i) \}^2 . \]
This decouples the optimization problem of minimizing
$R_n^{\mathcal{D}}(G)$ over %$\mathcal{G}$ into two separate problems of finding a
functions of the assumed form into two separate problems of finding a
sufficiently accurate estimate $\hat{U}$ of $U$, and finding individual
optimizations of the $R^{\mathcal{D}}_n(\hat u_k, g_k)$ along the directions $\hat u_k$.   
%{\color{blue} We might consider adding a blurb here about penalization and a definition (placed after presentation of the General Procedure) that helps make presentation of the main theorem clearer (Have tracked minor changes related to this below in purple).  Something along the lines of:}
{\color{black} We may also consider adding a penalty, $\mathcal P (G)$, to $R^{\mathcal{D}}_n(G)$.  A natural option is to impose smoothness assumptions on $g_k$, which one could do with a decomposable penalty of the form  $\mathcal P (G) = \sum_{k=1}^q \mathcal P_k (g_k) $ where $\mathcal P_k (\cdot)$ are semi-norms penalizing smoothness.  This results in minimization of $R^{\mathcal{D}}_n(v_k, g_k) +  \mathcal P_k (g_k)$ along the directions $v_k$.}
This observation suggests a two step fitting procedure:

\begin{itemize}
\item \textit{General Estimation Procedure:} Given data $y_1,\dots,y_n \in \mathbb{R}^p$ from model \eqref{nonparametricFactorData},
one obtains an estimate of $F(\cdot)$ as follows:  

\textit{Step 1.} Find an estimate $\tU=(\hat{u}_1 \ldots \hat{u}_q)$ of $U = (u_1 \ldots u_q)$. 

\textit{Step 2.} Plug in the estimate $\tU$ into $R_{n}
^{\mathcal{D}} (G)$ and find the corresponding minimizers of the $R_{n}
^{\mathcal{D}} (\hat u_k, g_k)$, or penalized versions thereof, denoted by $\hat{f}_1,\dots,\hat{f}_k$. 
\end{itemize}
Then the estimate $\hat{F}$ is given by
\[    \hat{F} =  \sum_{k=1}^q \hat{u}_k \hat{f}_k.    \]
If, for any vectors $\{v_1, \dots, v_q \}$, minimizing the $R^{\mathcal{D}}_n(v_k, g_k) +  \mathcal P_k (g_k)$ along the directions $v_k$ results in an identical smoothing procedure applied to the
data $v_k^{\T} y_i$ for $k=1,\dots,q$, we call the General Estimation Procedure \emph{direction invariant}.
%{\color{purple} If minimizing the $R^{\mathcal{D}}_n(v_k, g_k) +  \mathcal P_k (g_k)$ along any set of directions $v_k$ results in a common smoothing procedure applied to the data $v_k^{\T} y_i$, we call the General Estimation Procedure \emph{direction invariant}.}
% {\color{red} Mark, Fang finds this sentence unclear. Would you please rewrite it?}
This procedure is general, and leaves flexibility in model fitting while
provides an easy route to develop theory. In next subsection, we present the salient theoretical features of the general estimation procedure.

\subsection{Theoretical guarantees}
\label{Theoretical guarantees}

We first present a proposition which ensures that a reasonable $F:[0,1] \rightarrow \Rp$
has a singular value decomposition type representation and supports the form of function proposed in the paper, while its proof is deferred to the Section \ref{Appendixprooftheorem}. 
\begin{prop}  \label{svdRepresentation}
Suppose that $F:[0,1]^d \rightarrow \mathbb{R}^p$, which can be written as $F(\cdot) = (F_1(\cdot),
\dots,F_p(\cdot))^{\T}$, satisfies $F_k \in L^2[0,1]^d$ for $ 1\leq k\leq p $.  Then $F(\cdot)$
has a singular value type decomposition
\[   F = \sum_{k=1}^p \sigma_k  u_k \otimes v_k=\sum_{k=1}^p u_k \otimes f_k,    \]
where $ f_k= \sigma_k v_k $, $v_k$'s are orthonormal in $L^2[0,1]^d$, $u_k$'s are orthonormal in $\Rp$ and $\sigma_1\ge \sigma_2 \ge \ldots \ge \sigma_p\ge 0 $.
\end{prop}
\begin{remark}
By this proposition, if we further impose additional low rank assumptions, i.e. $ \sigma_k=0 $ for all $ q+1\leq k\leq p $, and smooth assumptions on $ f_k $'s, one can obtain the form of function \eqref{nonparametricFactorizationModel}. 
\end{remark}

Next, we propose the general theorems for the estimation procedure outlined above. The proofs of the theorems are deferred to the Section \ref{Appendixprooftheorem}.

%{\color{blue} Maybe use $Y^o$ to draw a clear distinction between oracle rotated and the estimated rotation; same notation for function estimates.  This distinction is important to the theorem.}
%Could also try changing to something like; 
We assume the design points $x_i$'s are independent and identically distributed with $x_i$ following uniform distribution on $ [0, 1]^d$, i.e. $x_i\sim \mathcal{U} [0, 1]^d$. We also assume $ z_i=(z_{i1}, \ldots, z_{ip})^{\T} \in \Rp $ are independent and identically distributed as $ \mathcal{N}(0,\sigma ^2I_p) $, which puts us in the domain of $p$ repeated nonparametric experiments. The assumption of uncorrelated Gaussian noise is commonly used to facilitate model exploration and theoretical development \citep{cai2012, donoho1994ideal, Tsybakov:2008:INE:1522486, Johnstone2017gsm} in the study of nonparametric experiments. In Section \ref{additionalsim} in the supplementary file, we have shown through simulations that our method works well without the Gaussian noise assumption, even without the independence assumption of $ z_i $ across $ 1\leq i\leq n $.

We first assume we have obtained a good estimate $\tU$ of $U$ so that $\| \sin \Theta (\tU ,U) \|$ is small, postponing discussion of how to obtain such an estimate $\tU$ to Theorem \ref{mainConsistencyfirststep}.   We then define the estimated rotated response as $ \hat{Y}^*_{ik}=\hat{u}_k^\T y_i $ for $ 1\leq k\leq q $ and $\hat Y^*_{\cdot k}=(\hat{Y}^*_{1k}, \dots, \hat{Y}^*_{1k})^{\T} $.  Further, we define the oracle rotated data $\{(Y^o_{ik}, x_i): 1\leq i\leq n\} $ as 
%\[  Y^*_{ik} = u_k^\T y_i = u_k^\T F(x_i)+ u_k^\T \epsilon_i=f_k(x_i) + \epsilon_{ik}, 
%\]
\[  Y^o_{ik}  \equiv u_k^\T y_i =f_k(x_i) + \epsilon_{ik},  \hspace{5pt} \text{ where } \hspace{5pt}  \epsilon_{ik} \equiv u_k^{\T}z_i.
\]
Note that the model assumptions guarantee that the  $\epsilon_{ik}$'s are independent and identically distributed $ \mathcal{N}(0, \sigma^2) $ for $ 1\leq i\leq n $ and $ 1\leq k\leq q $. 
%{\color{red} Above sentence edited from: where $\epsilon_{ik}$'s are assumed to follow independent and identically distributed $ \mathcal{N}(0, \sigma^2) $ for $ 1\leq i\leq n $ and $ 1\leq k\leq q $. }
%{\color{blue} Taking expectations shows that our assumptions guarantee $\E \epsilon_{il} \epsilon_{jk} = \sigma^2 \delta_{ij} u_l'u_k = \sigma^2 \delta_{ij}  \delta_{lk}$.  Given that the $\epsilon_{ik}$ are linear combinations of mean 0 Gaussians, and hence mean 0 Gaussians themselves, this means the $\epsilon_{ik}$ have the same distribution as the $z_{ik}$.  This is just rotational invariance of the Gaussian distribution and seems to be a point that isn't belaboured in the nonparametric literature.}
%{\color{blue} (Mark, suppose the variance of $ \epsilon_{ik} $ is $ \sigma_{k}^2 $, do we allow $ \sigma^2\neq 1 $? In addition, can we allow $ \sigma_{k}^2 $'s to be different for different $ k $?} 
Now assume $ L $ is a linear smoother and note that we may apply $L$ to the oracle rotated data $\{(Y^o_{ik}, x_i): 1\leq i\leq n\} $ to obtain nonparametric estimates $ \hat{f}^o_k $ for $ 1\leq k\leq q $. 
%and further define $ \hat{f}^o=(\hat{f}^o_1,\ldots, \hat{f}^o_q)^\T  $.   
If $L$ satisfies $ \max_{1\leq k \leq q} \E \| \hat{f}^o_k - f_k \|_n^2 \le C  n^{ -r } $ from the estimates $\hat{f}^o_k = L Y^o_{\cdot k}$, where $Y^o_{\cdot k} = (Y^o_{1k}, \dots, Y^o_{nk})^{\T}$, then we say $L$ attains the rate $r$.  
%{\color{red} Above sentence was edited from: We may then apply the linear smoother $L$ to the data $\{(\hat{Y}^*_{ik}, x_i): 1\leq i\leq n\} $ to obtain nonparametric estimators $ \hat{f}_k $ for $ 1\leq k\leq q $ and further define $ \hat{f}=(\hat{f}_1,\ldots, \hat{f}_q)^\T  $.   If $L$ satisfies $ \max_{1\leq k \leq q} \E \| \hat{f}_k - f_k \|_n^2 \le C  n^{ -r } $ from the estimates $\hat{f}_k = L Y^*_{\cdot k}$, where $Y^*_{\cdot k} = (Y^*_{1k}, \dots, Y^*_{nk})^{\T}$, then we say $L$ attains the rate $r$.}
%{\color{blue} For the theorem, it is only important to know how the linear smoother performs on the oracle rotated data.  From this, given a good estimate $\tU$ of $U$ (one satisfying the assumptions of the theorem), we can conclude the rates of the theorem.}
With these definitions in place, we have the following theorem.

%{\color{blue}(Mark, can you define $R_n(\hat{F})$?)}
%{\color{red}(This was contained in initial versions; there it was part of the presentation of the model and estimation goals.  Not sure where you would like to put it now?)}
{\color{black}
\begin{thm}
\label{mainConsistencyTheorem}
Suppose we have an estimate $\tU$ of $U$ satisfying 
\[ \E \left\{ \| \sin \Theta (\tU ,U) \|^4 \right\} \le C (p/n)^2 \]
and that the General Estimation Procedure is direction invariant, resulting in a bounded linear smoother $L$, $||L|| \le C$, attaining the rate $r$.  Then if $p=O(n)$ and $\max_k \|f_k \|_{\infty} \le B$, the corresponding estimator $\hat{F} = \sum_{k=1}^q \hat{u}_k \hat{f}_k$ formed from the General Procedure admits an error 
\[ \E \left\{ R_n(\hat{F}) \right\} \lesssim  q n^{-r} \max \left\{1, p n^{-(1 - r)} \right\}. \]  
Consequently,  as long as $p = O(n^{1 - r})$ we have $ \E \left\{ R_n(\hat{F}) \right\} \lesssim q n^{-r}$.
\end{thm} }

\begin{remark}
A main appeal of this theorem is that it is agnostic about the choices of $\tU$ and $L$.  In fact, given standard smoothness assumptions, there
is a vast range of literature on designing linear smoothers that attain the needed nonparametric rate.  Examples of linear smoothing include 
	regression in truncated basis function expansions, spline expansions, ridge penalized 
	variants of these, and reproducing kernel Hilbert space regression, see \cite{buja1989} for a summary. 
Thus the crux of applying this theorem, in most cases, will lie
in choosing estimates $\tU$ of $U$ and establishing rates for $\E \left\{ \| \sin \Theta (\tU
,U) \|^4\right\}$, which satisfy the assumptions of the theorem. 
\end{remark}

The next step is to find estimates $\tU$ of $U$ such that $\E \left\{ \| \sin \Theta (\tU,U) \|^4\right\}\le C (p/n)^2$ holds, allowing to take a step toward applying Theorem \ref{mainConsistencyTheorem}. Let $Y = (y_1 \ldots y_n) \in \mathbb{R}^{p\times n} $ be the response data matrix, $ \tilde{F}=(F(x_1) \ldots F(x_n)) \in \mathbb{R}^{p\times n} $ and $ Z=(z_1 \ldots z_n) \in \mathbb{R}^{p\times n}  $, one can write $Y \equiv \tilde{F} + Z$. 

%Due to the curse of dimensionality, guarantees for fixed and random design are quantitatively different as the sampling dimension is increased.   We thus present separate conditions in the following theorem.  

\begin{thm} 
\label{mainConsistencyfirststep}
Suppose that the components $f_k$'s of $f$ satisfy $f_k \in L^2[0, 1]^d$ and are bounded so that $\max_k \| f_k \|_{\infty} \le B$. %suppose that
%1. For \emph{\textrm{fixed design}}, $f_k \in L^2[0, 1]$ are of bounded variation with design points which are \emph{near equi-spaced}, so that the empirical cumulative density function of the $x_i$, $H_n$, satisfies $n \| H_n - x \|_{\infty} = O(1)$. 
%2. For \emph{\textrm{random design}}, $f_k \in L^2[0, 1]^d$ with design points that are uniformly distributed so that $x_i$'s are independent and identically distributed following $\textrm{Uniform} [0, 1]^d$, $d \ge 1$.
Then the top $q$ left singular vectors $\tU$ of $Y \equiv \tilde{F} + Z$ satisfy $\E \left\{ \| \sin \Theta (\tU ,U) \|^4 \right\} \le C (p/n)^2.$
{\color{black} Consequently, if the rest conditions of Theorem \ref{mainConsistencyTheorem} are satisfied, we may conclude that the General Estimation Procedure results in an estimate of $F$ with rate $ \E \left\{ R_n(\hat{F}) \right\} \lesssim q n^{-r}$. }
\end{thm}

\begin{remark}
Under the constraint $p \lesssim n^{1-r}$, this reduces to the standard rate of recovery for $q$ functions by $L$. In addition, when $ p = o(n)$, the risk remains $o(1)$. Theorem \ref{mainConsistencyTheorem} is proven by a natural decomposition of the estimation error together with an appeal to linearity.
\end{remark}

\begin{remark}
The constraint that $p \lesssim n^{1-r}$ arises primarily because the bound $\E \left\{ \| \sin \Theta (\tU ,U) \|^4 \right\} \le C (p/n)^2$ is the best that we can achieve based on \citet{cai2016rate}. Although beyond the scope of this paper, we note that there are
techniques available for sparsity constrained estimation of the singular value decomposition of a matrix \citep{witten2009penalized, kuleshov2013fast, ma2013sparse, yang2014sparse, lei2015sparsistency}; as in the case of $\ell_1$ penalized estimation of regression parameter, one might expect that these allow improvement to $\E \left\{ \| \sin \Theta (\tU ,U) \|^4 \right\} \le C (p^{\gamma}/n)^2$ for some $\gamma \in (0,1)$, or even possibly powers of $\E \left\{ \| \sin \Theta (\tU ,U) \|^4 \right\} \le C (\log p/n)^2$. Indeed, we feel that this would be an interesting avenue for future research. 
\end{remark}

\begin{remark}
%In Theorem \ref{mainConsistencyfirststep}, we consider two sampling mechanisms: $x_i$ are either non-random and nearly equi-spaced, i.e. $n \| H_n - x \|_{\infty} = O(1)$ considering univariate covariate $x_i$, or random with uniform distribution, i.e. $x_i$'s are independent and identically distributed following $\textrm{Uniform}[0,1]^d$.  The theorem states that if either of the two sampling mechanisms holds, we have $\E \left\{ \| \sin \Theta(\tU ,U) \|^4 \right\} \le C (p/n)^2$. This is proven by  taking advantage of the fact that under reasonable smoothness
Theorem \ref{mainConsistencyfirststep} states that under the uniform sampling mechanism of the design points $x_i$'s, we have $\E \left\{ \| \sin \Theta(\tU ,U) \|^4 \right\} \le C (p/n)^2$. This is proven by  taking advantage of the fact that under reasonable smoothness
assumptions, the inner product $\langle \cdot,\cdot \rangle_n$ and corresponding norm provide a good approximation to $\langle \cdot,\cdot \rangle_{L^2}$ and its corresponding norm.  This allows to conclude that the span of $\tilde{F} \equiv (F(x_1) \ldots F(x_n))$ is the
same as the span of $U$, then one can apply the results of \cite{cai2016rate} to derive
the desired bound on $\E \left\{ \| \sin \Theta (\tU ,U) \|^4 \right\}$.  
\end{remark}

%%%%%%%%%%%%%%%%%%%%
%%.  Examples. %%%%%%%%%%%
%%%%%%%%%%%%%%%%%%%%
We now provide some examples of standard linear smoothers and the rates they attain when used in the general estimation procedure. The details of these examples and the rates they attain are deferred to the Section \ref{lemmasMainExamples}. \\

\textit{Example 1} (\textrm{Local Polynomial Smoothing})
Suppose that the $f_k$'s lie in the \emph{H\"older class},  $\Sigma(\beta, L)$ and we perform a variant of local polynomial smoothing described in equation (\ref{LPSestimate}) to arrive at estimates $\hat{f}_k$.   Thus the $f_k$'s are $l = \lfloor \beta \rfloor$ times differentiable, where $\lfloor \beta \rfloor$ represents the largest integer strictly less than $\beta$, with the $l$th derivative $f^{(l)}$ satisfying
\[  |f^{(l)}(x) - f^{(l)}(y)| \le C | x - y |^{\beta - l},  \]
for all $x, y$ in the domain of interest.   {\color{black} With fixed parameters used in all directions, the procedure is direction invariant resulting in estimates at the data points $x_i$ that are linear in the data $\hat{f}_k = L \hat Y^*_{\cdot k}$.}  As we show in Section \ref{lemmasMainExamples},  $\| L \| \le C$ and $L$ attains rate $r = 2\beta / (1 + 2\beta) $.  If we further assume the $f_k$'s are bounded, $l \ge 1$, we may apply Theorem \ref{mainConsistencyfirststep} to find that the corresponding General Estimation Procedure yields an estimate $\hat F$ which satisfies $\E \left\{R_n(\hat F) \right\} \lesssim q n^{-2\beta / (1 + 2\beta) }$.  \\

\textit{Example 2} (\textrm{Truncated Series Expansions})
Suppose that the $f_k$'s lie in the Sobolev class of periodic functions of integer smoothness $\beta \ge 1$, denoted by $W^p(\beta, L)$. To define this class of functions, we start with the Sobolev class $W(\beta, L)$ defined by
\[ W(\beta, L) = \left\{ f \in L^2[0, 1] : f^{(\beta - 1)} \in \mathcal{C}[0, 1] \text{ and } \int_0^1 (f^{(\beta)})^2 \le L \right\} , \]
where $\mathcal{C}[0, 1]$ is the collection of absolutely continuous functions on $[0, 1]$.  The function class of interest, $W^p(\beta, L)$, is then defined by
\[ W^p(\beta, L)  =  \left\{ f \in  W(\beta, L): f^{(j)} (0) = f^{(j)}(1) \text{ for } j=0, 1, \dots, \beta - 1 \right\} .  \]
Fix the Fourier basis, where $\varphi_1 = 1$ and $\varphi_{2k}(x) = {2}^{1/2} \cos(2 \pi x)$, $\varphi_{2k + 1} (x) = {2}^{1/2} \sin(2 \pi x)$ for $k \ge 1$. {\color{black} If we estimate the $f_k$ by regressing the observations on the first $N \sim n^{1/(1 + 2\beta)}$  basis elements the resulting procedure is direction invariant, resulting in estimates at the data points $x_i$ that are linear in the data $\hat{f}_k = L \hat Y^*_{\cdot k}$. Further, we find that $\max_k \E \| \hat{f}^o_k - f_k \|_n^2 \lesssim n^{-2\beta / (1 + 2\beta) }$ and $\| L \| \le 1$. }  If we further assume that $f_k$'s are bounded, we may apply Theorem \ref{mainConsistencyfirststep} to arrive at an estimate $\hat F$ which satisfying $\E \left\{R_n(\hat F) \right\} \lesssim q n^{-2\beta / (1 + 2\beta) }$.  \\

\textit{Example 3} (\textrm{Reproducing Kernel Hilbert space regression})
Let $K(\cdot, \cdot)$ be a bounded, positive semidefinite kernel function $[0, 1]^2$, and $\mathcal{H} = \mathcal{H} (K)$ the associated reproducing kernel Hilbert space on $[0, 1]$ with norm $ \| \cdot \|_{ \mathcal{H}} $.  {\color{black} Suppose that the $f_k$'s lie in $\mathcal{H}$ with bounded norm $ \| f_k \|_{ \mathcal{H}} $, and we recover the $f_k$ by estimates $\hat{f}_k$ formed by the reproducing kernel Hilbert space penalized estimation
\[ \hat{f}_{k} = \arg  \min_{g \in \mathcal{H}} \left\{ \frac{1}{2}  \| \hat Y^*_{\cdot k} - g\|_n^2 + \kappa_n \| g \|_{\mathcal{H}}^2 \right\},  \]
which, by the representer theorem,  take the form of linear smoothers. If we define the kernel matrix $K = (K(x_i, x_j)/n)_{i, j = 1}^n$, we may write
\[ \hat{f}_{k} = n^{-1/2} \sum_{i=1}^n \hat{\alpha}_{i, k} K(\cdot, x_i)  \hspace{5pt} \text{ where }  \hspace{5pt} \hat{\alpha}_k = n^{-1/2}(K  + \kappa_n I_n )^{-1} \hat Y^*_{\cdot k}  \]
so that $( \hat{f}_{k} (x_1) , \dots, \hat{f}_{k} (x_n) )^{\T}  = K (K  + \kappa_n I_n )^{-1} \hat Y^*_{\cdot k}   = L \hat Y^*_{\cdot k} $
is linear in the data, with $L = K (K  + \kappa_n I_n )^{-1}$ and $\| L \| \le 1$ and the procedure is direction invariant.}  Suppose that the eigenvalues $\lambda_i$ in the eigendecomposition of $K$ scale as $\lambda_i \sim i^{-\alpha}$, $\alpha > 1$.  Then 
%under conditions particular to design %{\color{blue} (Mark, can you illustrate ``conditions particular to design" in detail?)} 
%and 
with properly chosen $\kappa_n$, we have $\max_k \E \left( \| \hat{f}_k - f_k \|_n^2 \right) \lesssim n^{-\alpha / (1 + \alpha) }$ and consequently an estimate $\hat F$ which satisfies $\E \left\{R_n(\hat F) \right\} \lesssim q n^{-\alpha / (1 + \alpha) }$.  

%{\color{blue} (Mark, the notation in Example 3 and Section 2.4 may need to make consistent, such as the kernel matrix etc. One thing to note is that Biometrika does not allow bold font. )}

\subsection{Implementation and parameter tuning}\label{estimation}
%Motivated by the theoretical development in Section \ref{Theoretical guarantees}, we propose the following specific estimation procedure which is implemented in our numerical studies. In particular, the estimator obtained by this procedure enjoys the theoretical properties given in Theorems \ref{mainConsistencyTheorem} and \ref{mainConsistencyfirststep}. 

%\textit{(Specific Estimation Procedure)} Given data $Y=(y_1 \ldots y_n) $,
%one can obtain an estimate of $F(\cdot)$ by the following two steps: 

%\textit{Step 1.} Find the top $q$ left singular vectors $U$ of $Y$, denote by $\tU=(\hat{u}_1 \ldots \hat{u}_q)$. 

%\textit{Step 2.} For each $ 1\leq k\leq q $, we solve
%\begin{equation}\label{estimationprocedure}
%\hat{f}_{k} = \arg  \min_{g \in \mathcal{H}} \left\{ \frac{1}{2}  \| \hat Y^*_{\cdot k} - g\|_n^2 + \lambda \| g \|_{\mathcal{H}}^2 \right\}
%\end{equation}
%where we assume the nonparametric function $ f_k(\cdot) $ lies in a functional space $ \mathcal{H} $ generated by a positive definite kernel function $K(\cdot,\cdot)$. Denote the estimates by $\hat{f}_k$, then the estimate $\hat{F}$ is given by $\hat{F} =  \sum_{k=1}^q \hat{u}_k \hat{f}_k$. 

%Step 2 is based on the Reproducing Kernel Hilbert space regression. 
To be specific, we adopt the reproducing kernel Hilbert space procedure in Example 3. The reproducing kernel Hilbert space method has been applied to various nonparametric/functional regression models  \citep{lin2006component, yuan2010reproducing, zhang2011linear, du2014penalized, sang2018sparse} with straightforward implementation available in software such as R.

%By the Representer theorem \citep{kimeldorf1971some}, the solution for \eqref{estimationprocedure} can be expressed as 
%\begin{equation*}
%\hat{f}_{k}(\cdot) = n^{-1/2} \sum_{i=1}^n \hat{\alpha}_{i, k} K(\cdot, x_i),  
%\end{equation*}
%where $ \hat{\alpha}_k = (\hat{\alpha}_{1, k}, \ldots, \hat{\alpha}_{n, k})^{\T}=n^{-1/2}(K  + \lambda  I_n )^{-1} \hat Y^*_{\cdot k} $. 

For convenience, we use the Gaussian radial basis kernel function defined as $ K(x_1, x_2)$ $=\exp (-\|x_1-x_2\|^2/\rho )$. There are three tuning parameters involved in our estimation procedure: the number of retained dimension $q$ in the first-step estimation, the scale parameter $\rho$ in Gaussian radial basis function and the regularized parameter $\kappa_n$ in the second-step estimation. To select these parameters, we first fix the dimension at $ q $, and tune both $\rho$ and $\kappa_n$. In an ideal scenario, we want to tune both $\rho$ and $\kappa_n$ on a two dimensional fine grid, say using cross validation, however, this substantially increases the computation cost. Our preliminary studies show that the estimator is quite robust to the choice of $ \rho $. Therefore, to reduce the computational effort, we set $\rho $ as the median of $\{1\leq i< j \leq n: \|x_i -x_j\|^2\} $, denoted it by $ \rho^M$, previously suggested by \cite{gretton2012kernel, kong2016testing}.  For the regularization parameter $ \kappa_n $, we adopt the 10-fold cross validation procedure proposed in \citet{pahikkala2006fast}, and denote the selected parameter by $ \kappa_n^{(q)} $ when the dimension is fixed at $ q $. Let $ \hat {F}^{(q)} $ be the corresponding estimator of $ F $ using the retained dimension $q$ with tuning parameters $ \rho^M$ and $ \kappa_n^{(q)} $, and $ q $ can be chosen by minimizing
\begin{eqnarray}
\label{aicstar}
\AIC(q) = \log\left\{ V(q,\hat {F}^{(q)})   \right\} + 2  \left( \frac{q}{n}\right),
\end{eqnarray}
where $V(q,\hat {F}^{(q)} ) = \frac{1}{2n} \sum_{i=1}^n \|y_i-\hat {F}^{(q)} (x_i)\|_2^2$. %Let $ q^* $ be the optimal $ q $ obtained from minimizing $\AIC$, then the optimal tuning parameters are given by $ \{q^*, \kappa_n^{(q^*)}, \rho^M\} $.  

\section{Numerical Examples}\label{numerical}
\subsection{Simulation study}\label{simulation}
In this subsection, we perform simulation study to evaluate the performance of the proposed method. The $U=(u_1 \ldots u_q) $ is generated by orthonormalizing a $p\times q $ matrix with all elements being independent and identically distributed standard normal. The $x_i$'s are independently generated from uniform distribution on $[0, 1]$. The $f_k$'s are independently generated from a zero mean Gaussian process with compactly supported covariance function $C$,
\[   C_{\alpha,\beta}(s,t)  =  \beta  \max\{0,(1-r)^5\}  (8 r^2 + 5r + 1),  \]
where $r \equiv r_{\alpha}(s,t) = |s - t|/\alpha$, see \cite{rasmussen2006gaussian} for details. We set $\alpha = 0.5$ and $\beta = 15$. The error $ z_{ij} $ are generated independent and identically distributed from standard normal for $ 1\leq i\leq n $ and $ 1\leq j\leq p $. The response $ y_i $ is obtained by $ y_i=\sum_{k = 1}^q  f_k(x_i) u_k+z_i$. A suitable comparison would be conducted against individual nonparametric regression of $ Y $ on $x$ in a curve-by-curve manner.  In particular, we compare with the method that fits the $j$th component of $ Y $ on $ x$ nonparametrically for each $ 1\leq j\leq p $. For fair comparison, we also use the reproducing kernel Hilbert space regression with radial basis function kernel for curve-by-curve nonparametric recovery as well. For the tuning parameters, we use the selection method described in Section \ref{estimation}. Specifically, for each curve nonparametric regression, we set the scale parameter $ \rho $ in radial basis function kernel as the median of $\{1\leq i< j \leq n: \|x_i -x_j\|^2\} $, and select the regularization parameter $ \kappa_n $ using the 10-fold cross validation. For nonparametric principal subspace regression, we report the estimation error $\| F - \hat{F} \|_n^2=\frac{1}{n}\sum_{i=1}^n \|F(x_i)-\hat{F}(x_i)\|^2 $ and the estimated dimension $ \hat q $ selected by the $\AIC$. For curve-by-curve nonparametric recovery, we only report the estimation error since the procedure fits each curve individually. We consider different combinations of $ (n,p,q) $, and for each combination, we perform 100 Monte Carlo studies. 
{ \begin{table}[htbp]
\centering
\caption{Simulation results: the average estimation errors for our nonparametric principal subspace regression method (``NPSR error'') and the curve-by-curve nonparametric regression (``Nonparametric error''), and their associated standard errors in the parentheses are reported. The selected $q^*$ by $\AIC$ is also reported. The results are based on 100 Monte Carlo repetitions. }
\newsavebox{\tableboxa}
\begin{lrbox}{\tableboxa}
\begin{tabular}[l]{cccccccccccc}

\hline \hline
$n$ &$q$& $p$& NPSR error  &   Nonparametric error  & $q^*$ \\
\\

&  & 10 & 0.535 (0.038)  & 0.902 (0.039)  & 2.000 (0.000)   \\  
     &2& 20 & 0.708 (0.046)  & 1.537 (0.038)  & 2.000 (0.000)  \\  
     &  & 40 & 1.069 (0.051)  & 2.538 (0.040)  & 1.980 (0.014)   \\  
%     & 2 & 60 & 1.356 (0.035)  & 3.497 (0.050)  & 1.980 (0.014)   \\  
 %    & 2 & 80 & 1.670 (0.029)  & 4.332 (0.059)  & 1.950 (0.022)   \\  
  128&\\ 
     &  & 10 & 1.208 (0.088)  & 1.293 (0.066)  & 3.830 (0.038)  \\  
     & 4 & 20 & 1.430 (0.078)  & 1.821 (0.039)  & 3.850 (0.038)  \\  
     &  & 40 & 2.147 (0.059)  & 3.237 (0.062)  & 3.600 (0.049)  \\  
     \\
     \hline\\
 %    & 4 & 60 & 2.757 (0.058)  & 4.285 (0.068)  & 3.560 (0.050)  \\ 
  %   & 4 & 80 & 3.330 (0.035)  & 5.488 (0.079)  & 3.500 (0.050)   \\  
 %    & 2 & 10 & 0.346 (0.020)  & 0.617 (0.040)  & 2.000 (0.000)  \\  
     &  & 20 & 0.451 (0.034)  & 0.890 (0.024)  & 2.000 (0.000)  \\  
     & 2 & 40 & 0.685 (0.059)  & 1.464 (0.028)  & 1.990 (0.010)  \\  
     & & 60 & 0.784 (0.024)  & 2.021 (0.036)  & 2.000 (0.000)  \\  
 %    & 2 & 80 & 0.948 (0.024)  & 2.532 (0.041)  & 1.970 (0.017)  \\  
  %   & 4 & 10 & 0.669 (0.036)  & 0.758 (0.041)  & 3.960 (0.028)   \\  
   256&\\ 
     &  & 20 & 0.856 (0.045)  & 1.158 (0.046)  & 3.920 (0.031)  \\  
     & 4 & 40 & 1.344 (0.074)  & 1.783 (0.034)  & 3.810 (0.039)  \\  
     & & 60 & 1.575 (0.047)  & 2.429 (0.033)  & 3.740 (0.044)  \\  
     \\
     \hline\\
 %    & 4 & 80 & 1.946 (0.064)  & 3.042 (0.042)  & 3.720 (0.047)  \\  
  %   & 2 & 10 & 0.275 (0.054)  & 0.372 (0.019)  & 2.000 (0.000)  \\  
  %   & 2 & 20 & 0.273 (0.014)  & 0.552 (0.018)  & 2.000 (0.000) \\  
     &  & 40 & 0.376 (0.032)  & 0.867 (0.017)  & 2.000 (0.000)  \\  
     & 2 & 60 & 0.478 (0.025)  & 1.172 (0.020)  & 2.000 (0.000)  \\  
     &  & 80 & 0.512 (0.011)  & 1.436 (0.022)  & 2.000 (0.000)  \\  
      512&\\ 
%    & 4 & 10 & 0.490 (0.032)  & 0.647 (0.046)  & 4.010 (0.010)  \\  
 %    & 4 & 20 & 0.718 (0.050)  & 0.811 (0.031)  & 3.980 (0.014)  \\  
     &  & 40 & 0.876 (0.057)  & 1.189 (0.033)  & 3.910 (0.029)  \\  
     & 4 & 60 & 1.073 (0.054)  & 1.563 (0.031)  & 3.850 (0.036) \\  
     &  & 80 & 1.165 (0.031)  & 1.857 (0.028)  & 3.870 (0.034) \\  
      \\
   \hline \hline
\end{tabular}
\end{lrbox}
\label{tab1:simu1} 
\scalebox{1}{\usebox{\tableboxa}}
\end{table} }
From Table \ref{tab1:simu1}, one can see that our method outperforms the curve-by-curve nonparametric regression for all cases. For fixed $n$ and $q$, the recovery results from nonparametric principal subspace regression tend to improve at a faster rate as $p$ increases. Besides, one sees that the $\AIC$ is capable of choosing $ q^* $ close to the true value $ q $. We plot the estimates of the first four components $ f_1 $, $f_2$, $ f_3$ and $ f_4 $ from a randomly selected Monte Carlo run in the case of $ (n, p, q)=(256,40,4)$ in Figure \ref{simPlots}, showing good recovery of each nonparametric component. 
\begin{figure}[htbp]
  \subfigure[]{\includegraphics[height=2in,width=3in]{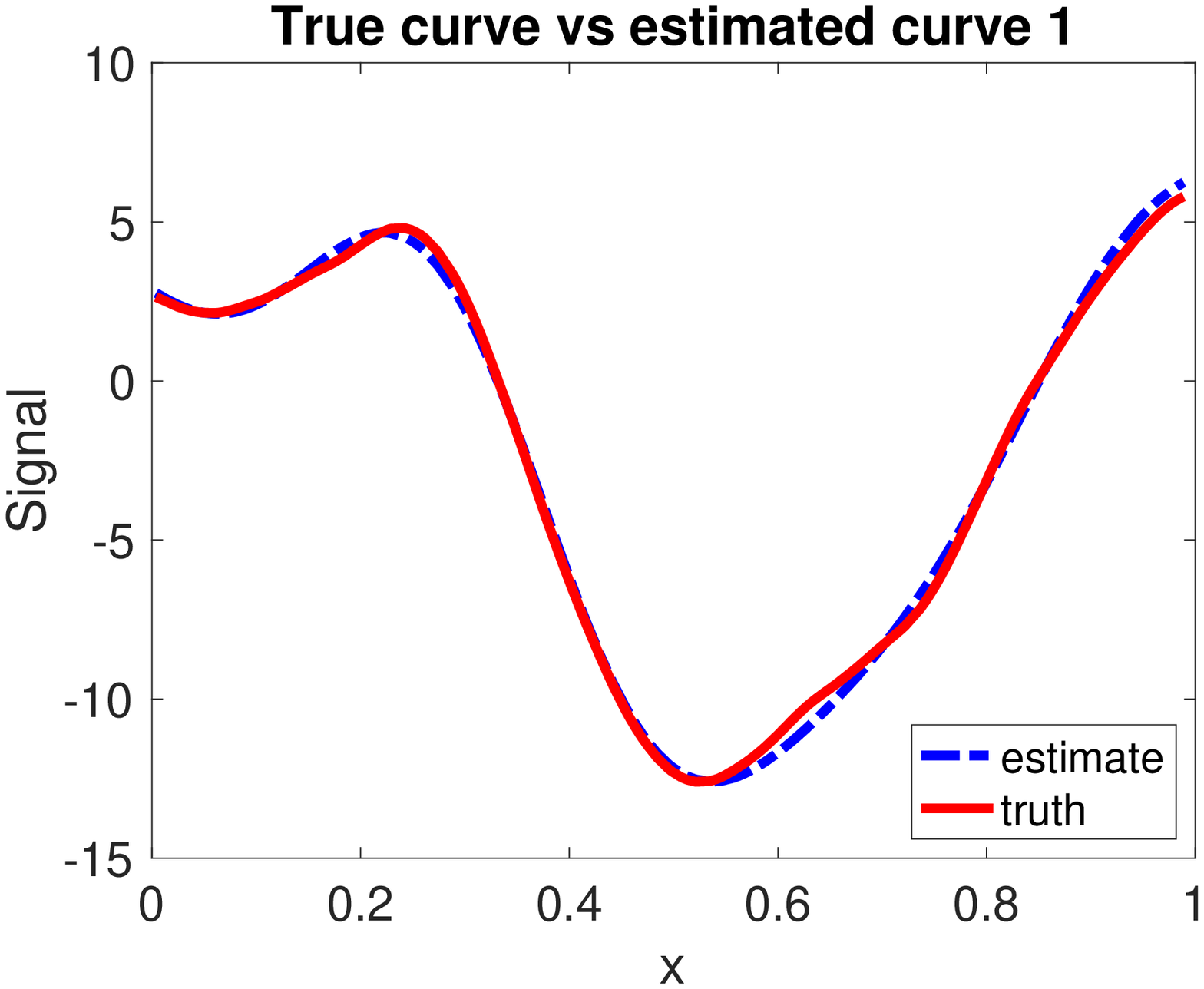}}
  \subfigure[]{\includegraphics[height=2in,width=3in]{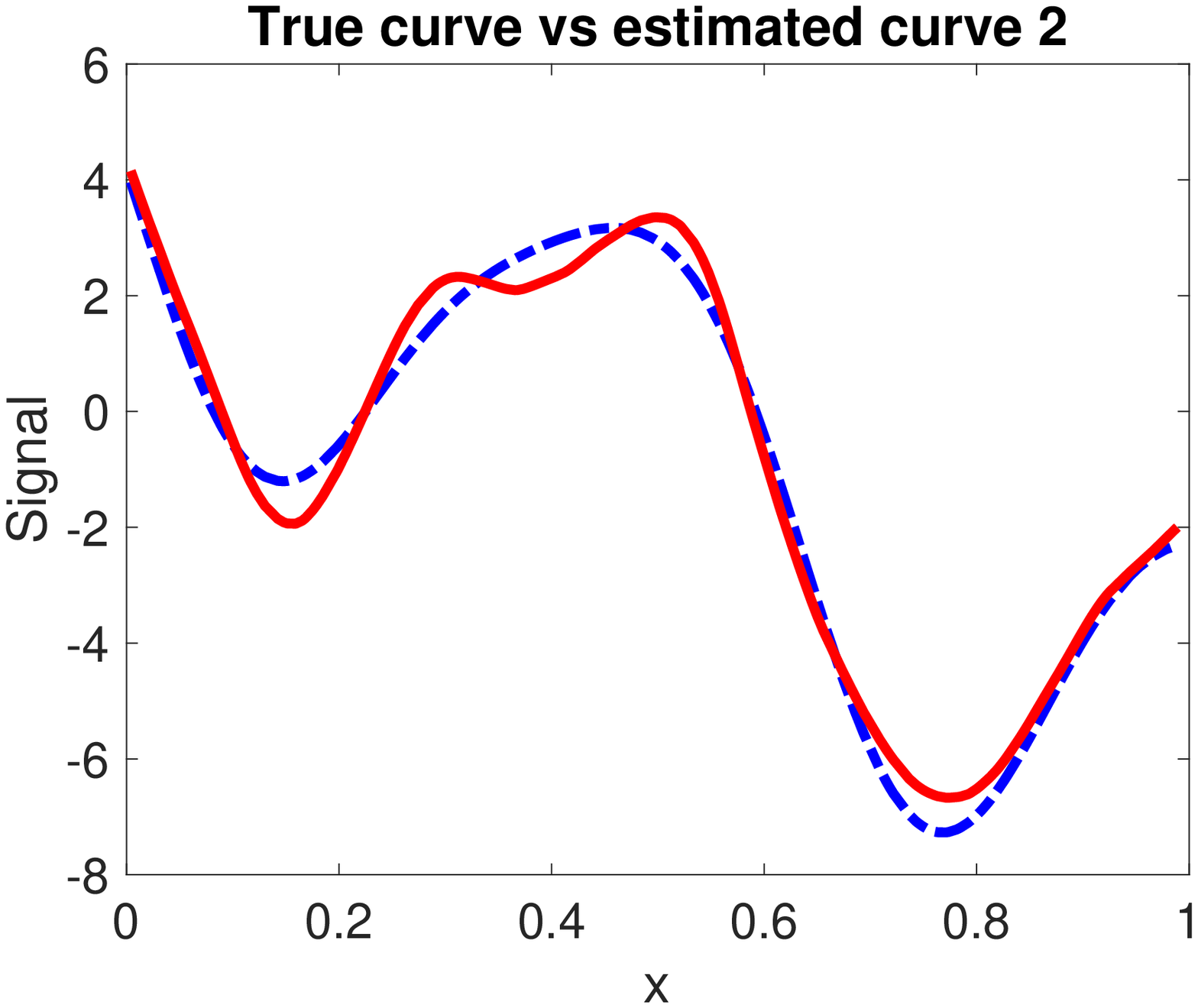}}
  \subfigure[]{\includegraphics[height=2in,width=3in]{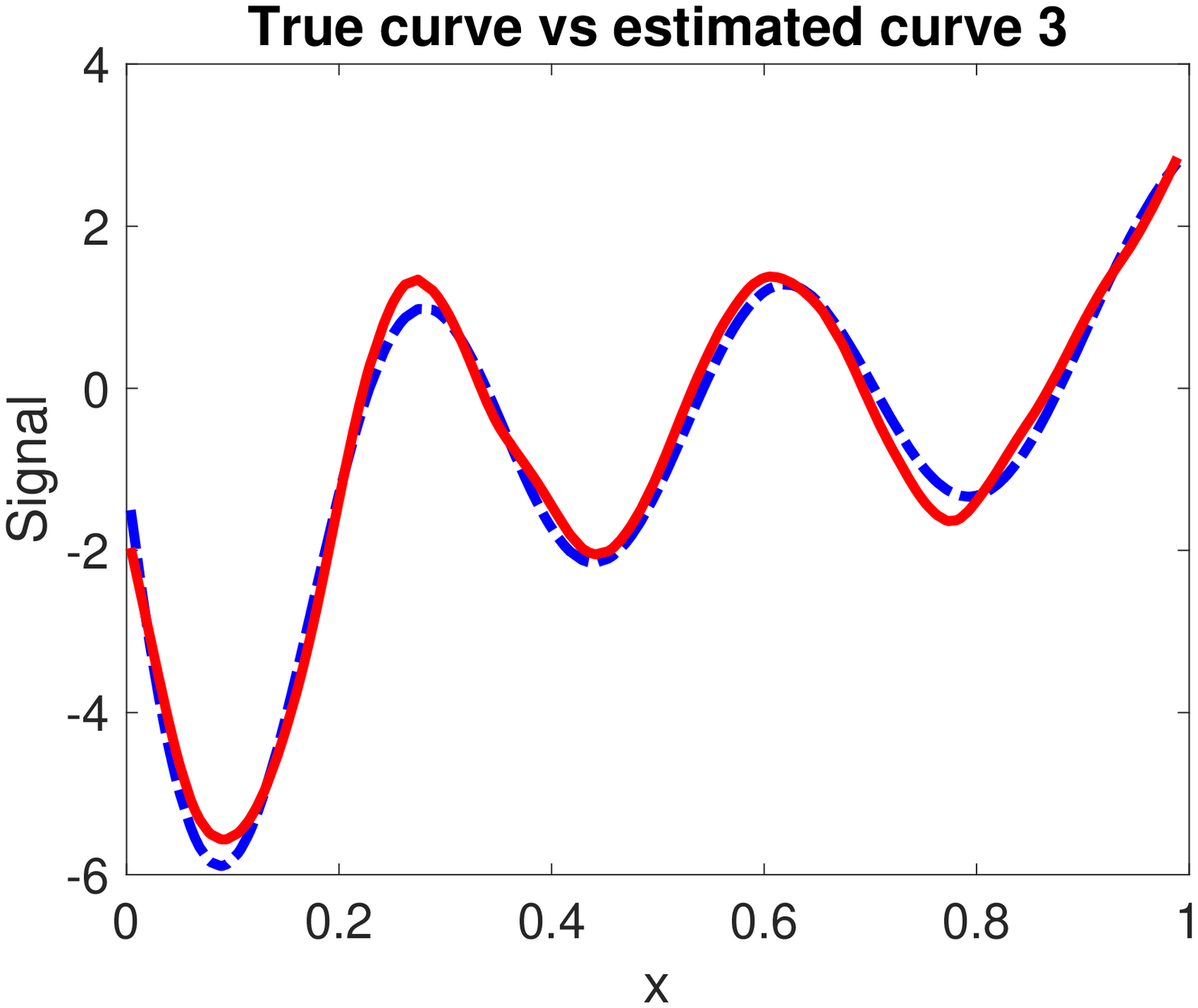}}
  \subfigure[]{\includegraphics[height=2in,width=3in]{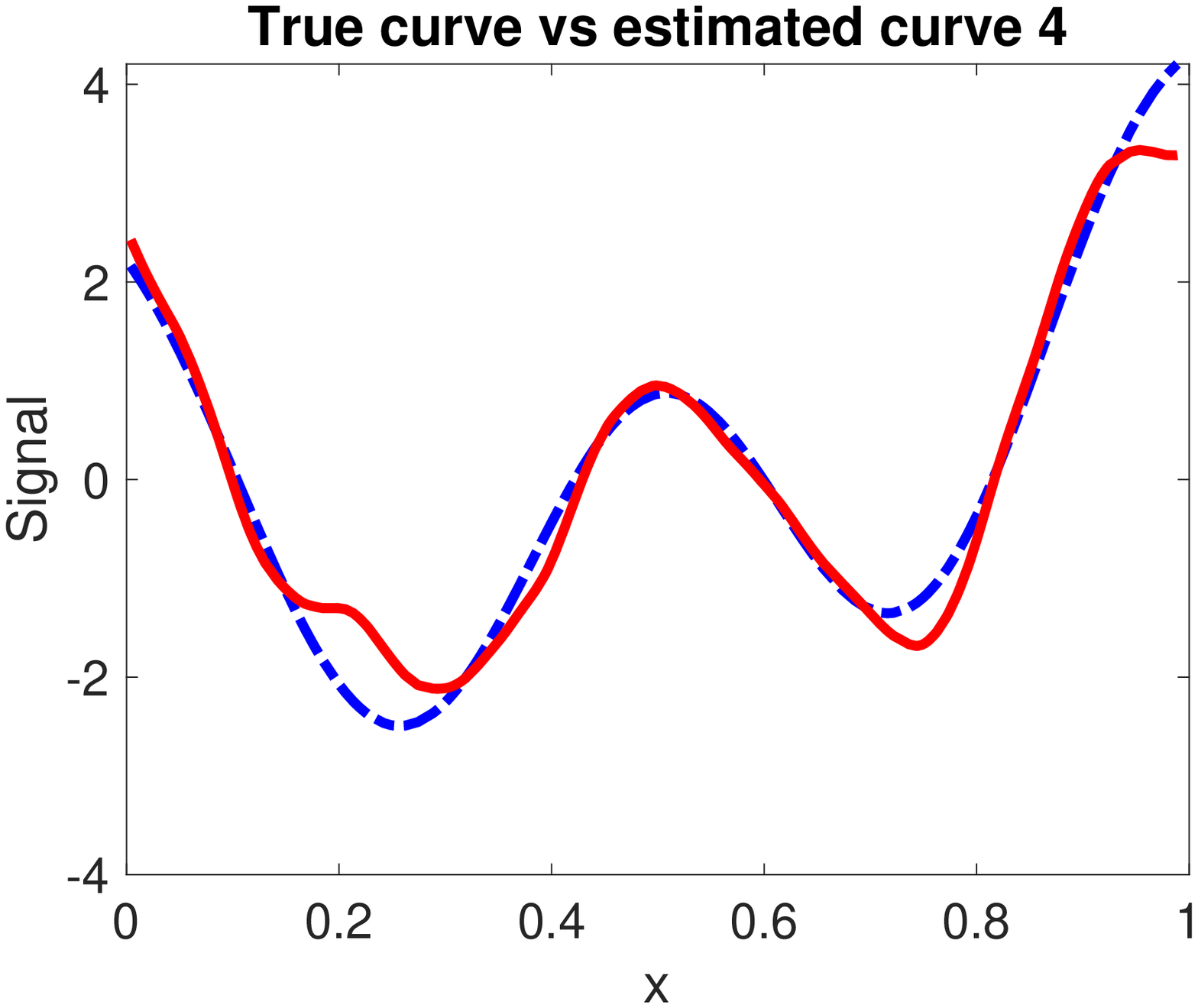}}
\caption{Simulation results: Panels (a)--(d) plot the estimated $\hat{f}_1$, $\hat{f}_2$, $\hat{f}_3$, $\hat{f}_4$ from our nonparametric principal subspace regression and their corresponding true functions $ f_1 $, $ f_2$, $ f_3$ and $ f_4 $ from a randomly selected Monte Carlo run in the case $ (n,p,q)=(256,40,4)$.}
\label{simPlots}
\end{figure}
We have performed additional simulation studies, where the error $ z_i $'s are correlated across $ 1\leq i\leq n $ and the components within each $ z_i $ are also correlated, reflecting more realistic settings in real applications. The additional simulation results are included in Section \ref{additionalsim} of the supplementary file, and indicate that our method still outperforms curve-by-curve nonparametric regression.

\subsection{Application to an electroencephalogram  study}\label{EEGdata}
We apply the proposed method to an electroencephalogram dataset, which is available at 
%https://archive.ics.uci.edu/ml/datasets/EEG+Database. 
\url{https://archive.ics.uci.edu/ml/datasets/EEG+Database}. 
The data were collected by the Neurodynamics Laboratory
and contain 122 subjects. Researchers measured the voltage values from 64 electrodes placed on each subject's scalps sampled at 256 Hz for 1 second. As electroencephalogram data are notoriously noisy while there are known to be strong correlations between different electrodes, the data from each subject may be considered as a sample from model (\ref{nonparametricFactorData}). In particular, for each subject, we obtain a data matrix $ Y=(y_1 \ldots y_n)\in \mathbb{R}^{p\times n} $ with $ p=64 $ and $ n=256 $. We fit the nonparametric principal subspace regression to the data matrix obtained from each subject. The average retained dimension selected by the proposed $\AIC$ among these 122 subjects is $6.959$ with standard error $0.113$. We have plotted the estimates of the first three functional components $ f_1 $, $f_2$ and $ f_3$ from a randomly selected subject in Figure \ref{eegPlots} (b)-(d). The curves show clear nonlinear patterns along the covariate time, which can not be captured by either classical factor model, singular value decomposition or multivariate response linear regression.  

\begin{figure}[htbp]
 \subfigure[]{\includegraphics[height=2in,width=3in]{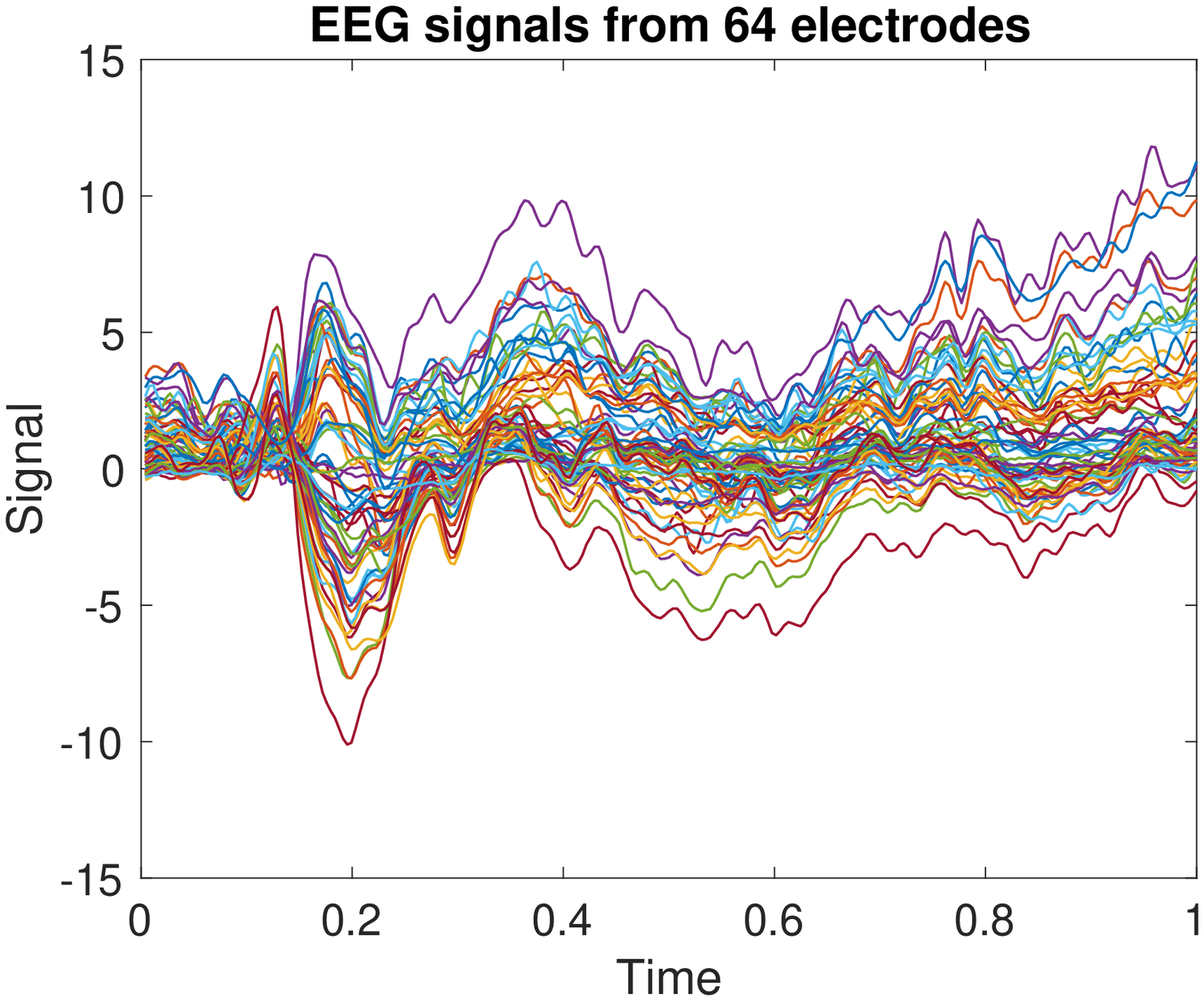}}
 \subfigure[]{\includegraphics[height=2in,width=3in]{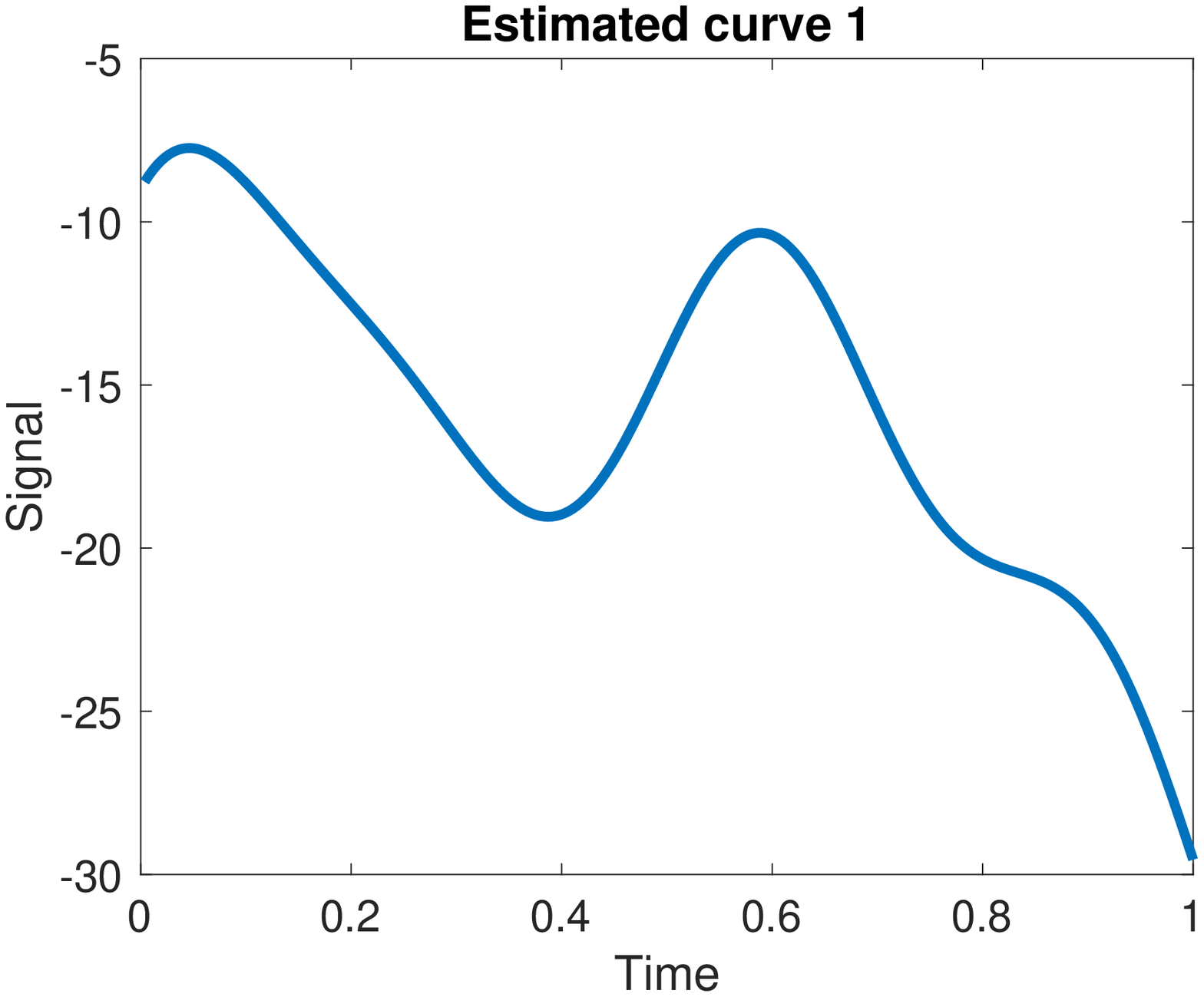}}
 \subfigure[]{\includegraphics[height=2in,width=3in]{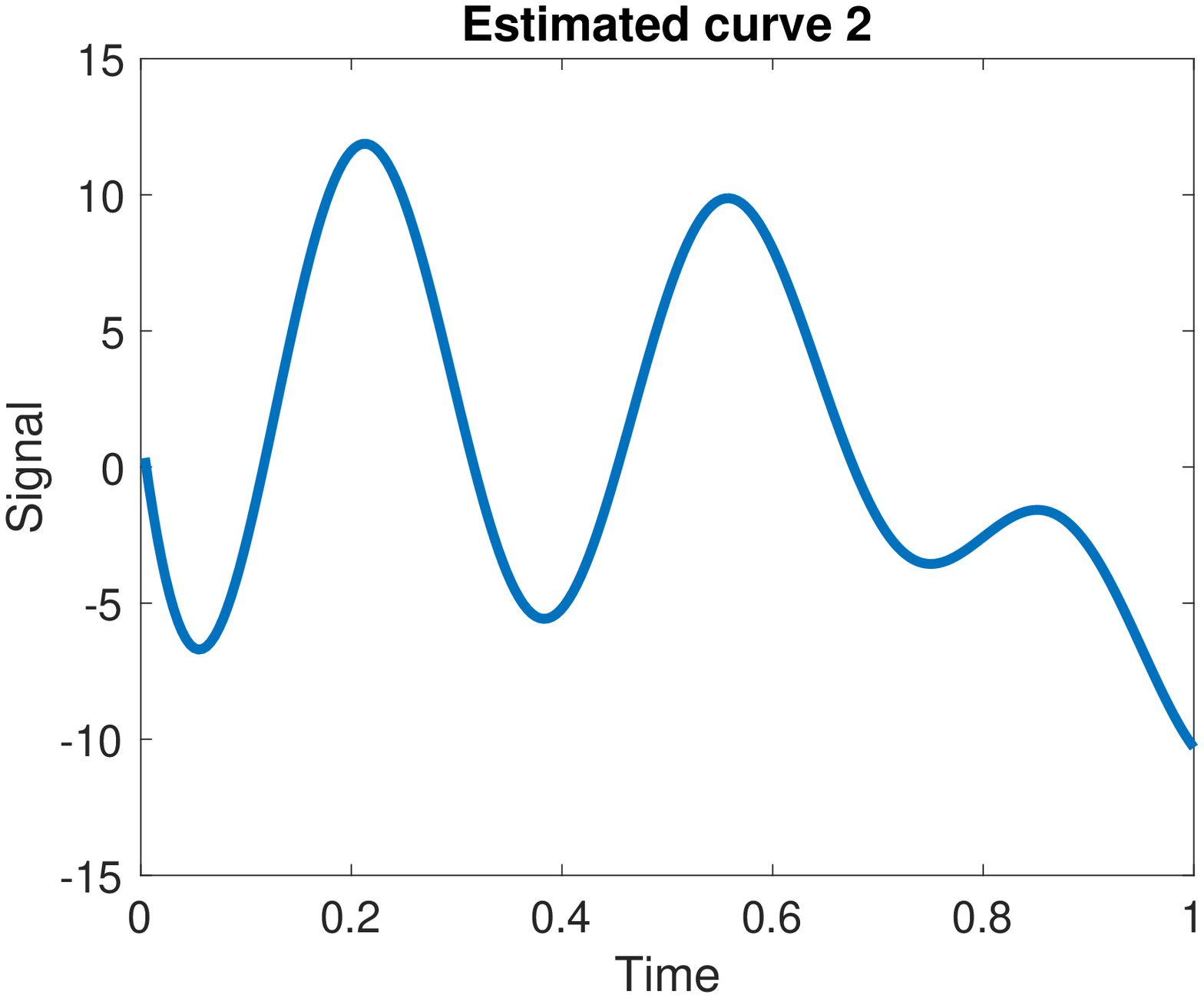}}
 \subfigure[]{\includegraphics[height=2in,width=3in]{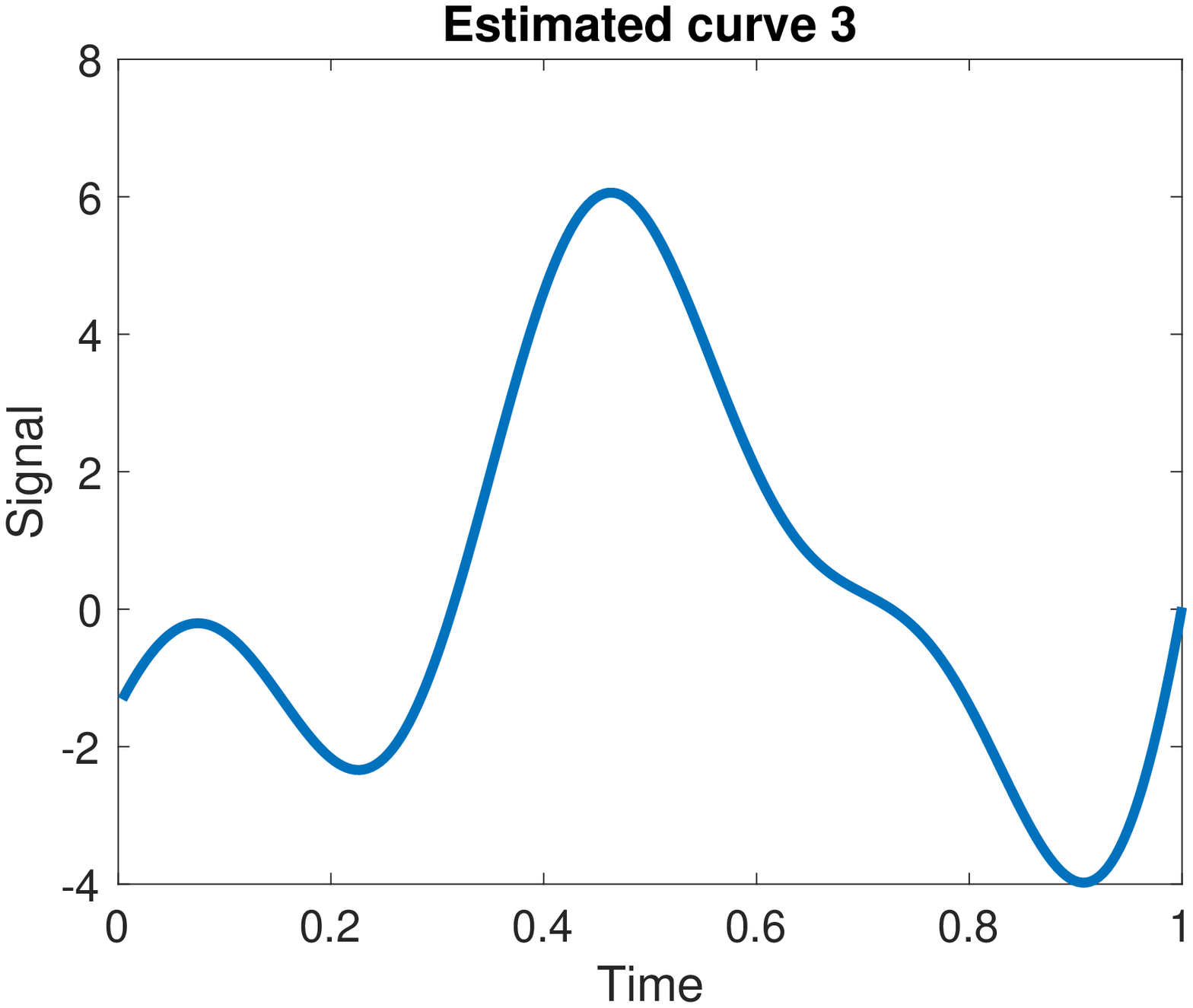}}
\caption{Real data results: Panel (a) plots the electroencephalogram signals detected from 64 electrodes of the scalp of one randomly selected subject. Panels (b)-(d) plot the estimates of the first three functional components $\hat{f}_1$, $\hat{f}_2$ and $\hat{f}_3$ from a randomly selected subject. }
 \label{eegPlots}
\end{figure}

To compare  the prediction performance, we also fit curve-by-curve nonparametric regression to the signals obtained from each of the $64$ electrodes. For each subject, we randomly reserve $10\%$ of data as the test set: $\mathcal{S}_{\rm{test}} \subseteq \{ 1,\ldots, 256\}$ such that $\left| \mathcal{S}_{\rm{test}} \right| / 256 \approx 10 \%$, while using the rest as the training set, and report the prediction errors $\left| \mathcal{S}_{\rm{test}} \right|^{-1} \left\{\sum_{i \in  \mathcal{S}_{\rm{test}}} (\|Y_{i} - \hat{F}(x_i) \|^2/64) \right\}$ for both approaches. The average prediction error for nonparametric principal subspace regression over the 122 subjects is $1.984$ with standard error $0.156$, while that obtained by the curve-by-curve nonparametric regression is $2.344$ with standard error $0.153$. %The results show that the proposed method achieves a higher prediction accuracy compared with curve-by-curve nonparametric regression. 

\section{Proofs of Main Theorems}\label{Appendixprooftheorem}
%\subsection{Notation}\label{notationMain}
We first introduce the notations used in the rest of this paper, some of the notations have been introduced before. Recall we have proposed a nonparametric model $ y_i  = \sum_{k = 1}^q  f_k(x_i) u_k  + z_i \in \Rp $, where $ z_i=(z_{i1}, \ldots, z_{ip})^{\T} $ follows independent and identically distributed $ \mathcal{N}(0,\sigma ^2I_p) $. Without loss of generality, in the proof we assume $ \sigma^2=1 $. 
%{\color{blue} (Mark, double check whether we can assume $ \sigma^2=1 $.)}

Let $Y = (y_1 \ldots y_n) \in \mathbb{R}^{p\times n} $ be the response data matrix, $ \tilde{F}=(F(x_1) \ldots F(x_n))$ $\in \mathbb{R}^{p\times n} $ and $ Z=(z_1 \ldots z_n) \in \mathbb{R}^{p\times n}  $, one can write $Y \equiv \tilde{F} + Z$. Let $ U=(u_1 \ldots u_q)\in \mathbb{R}^{p\times q} $, and $\tU=(\hat{u}_1 \ldots \hat{u}_q)\in \mathbb{R}^{p\times q}$ the estimate of $ U $. Define $ \tilde{f}_k=(f_k(x_1),\ldots, f_k(x_n))^{\T} \in \mathbb{R}^{n} $ for $ 1\leq k\leq q $ and $ \tilde{f}=(\tilde{f}_1 \ldots \tilde{f}_q)^{\T} \in \mathbb{R}^{q \times n} $ so that $\tilde F = U \tilde f$. We further define $ \tilde{\hat{f}}_k=(\hat{f}_k(x_1),\ldots, \hat{f}_k(x_n))^{\T}\in \mathbb{R}^{n} $ for $ 1\leq k\leq q $ and define $ \tilde{\hat{f}}=(\tilde{\hat{f}}_1 \ldots \tilde{\hat{f}}_q)^{\T} \in \mathbb{R}^{q \times n} $ so that we may write $\tilde{ \hat F} \equiv (\hat F(x_1) \ldots \hat F(x_n))  \in \mathbb{R}^{p\times n} $ as $\tilde{ \hat F} = \hat U \tilde{\hat{f}}$.

%\subsection{Proof of Main Propositions and Theorems}\label{prooftheorem} 
\begin{proof}[Proof of Proposition \ref{svdRepresentation}] 
We may consider $F$ as an operator $F : L^2[0,1]^d \rightarrow \Rp$ mapping $g \in L^2[0,1]^d$ to
\[  Fg = (\langle F_1,g \rangle_{L^2}, \dots ,\langle F_p,g \rangle_{L^2})  \in \Rp.  \]
In this case we note that under appropriate inner products
\[   F = \sum_{k=1}^p e_k \otimes F_k    \hspace{5pt}\text{ and }\hspace{5pt}   F^*F g
= \sum_{k=1}^p \langle F_k,g \rangle F_k  .\]
Thus $F^*F$ is of finite rank and hence compact.  As it is also symmetric, it has an
eigendecomposition
\[    F^*F  = \sum_{k=1}^{\infty}  \lambda_k^2 v_k \otimes v_k       \]
with at most $p$ of the $\lambda_k \ne 0$ and $v_k$ forming an orthonormal basis of
$L^2[0, 1]^d$.   We order so that the nonzero $\lambda_k$ lie in the first $p$ indices and are
increasing.  Now note that we may write
\[    g  = \sum_{k=1}^{\infty}  \langle g , v_k \rangle_{L^2} v_k  . \]
Now, setting $r_k  = F v_k$ we have that $\langle r_m , r_n \rangle_{L^2} = \langle v_m  ,
F^*F v_n \rangle_{L^2} =  \lambda_n^2 \langle v_m , v_n \rangle_{L^2} = \lambda_n^2 \delta_{mn}$.
Hence, the $r_k$ are orthogonal and at most $p$ of them are nonzero.  Setting $\sigma_k
= 1/\lambda_k$ for $\lambda_k > 0$ and $\sigma_k = 0$ otherwise, we set $u_k =  \sigma_k
r_k $ and note that since $\sigma_k = 0$ for the nonzero $r_k$, we  may write
\[    F g  =  \sum_{k=1}^{\infty}   \langle g , v_k \rangle_{L^2}  F v_k  = \sum_{k=1}^{p}
\sigma_k  \langle g , v_k \rangle_{L^2}  u_k  = \left(   \sum_{k=1}^p  \sigma_k u_k \otimes
v_k   \right) g.     \]
Hence under the appropriate inner products $F$ has the claimed representation.
\end{proof}

\begin{proof}[Proof of Theorem \ref{mainConsistencyTheorem}] 
%{\color{blue} (change the notations $ \tilde{F} $, $ \tilde{f} $, $\tilde{U}$, $ \hat{\tilde{f}} $, etc accordingly. The dimensions of some vectors here also look suspicious. This part needs a bit careful editing.)}
Given the definitions in the paper and at the outset of the supplement, we see that we may write $Y^o_{\cdot k} = Y^{\T} u_k$ and $\hat Y^*_{\cdot k} = Y^{\T} \hat u_k$ (both in $\mathbb{R}^n$).  Direction invariance combined with the assumptions of the theorem guarantee there is a linear smoother $L$, $||L|| \le C$ such that $\tilde{ \hat f}_k  = L \hat Y^*_{\cdot k}$.  Consequently $\tilde{ \hat f} = \hat U^{\T} Y L^{\T}$ and $\tilde{ \hat F} = \hat U \hat U^{\T} Y L^{\T}$ so that, using $\tilde F = U \tilde f$, we may decompose $\tilde F - \tilde{ \hat F}$ as 
%$\tilde F - \tilde{ \hat F} = U \tilde f -  \hat U \hat U^{\T} Y L^{\T} = U \tilde f - UU^\T  YL^\T  + (UU^\T  - \tU \tU^\T  ) Y L^\T $ and hence write
\[  \tilde F - \tilde{ \hat F}  = \underbrace{U \tilde f - UU^\T  YL^\T }_{\text{I}} + \underbrace{(UU^\T  -\tU \tU^\T  ) Y
L^\T }_{\text{II}}  \]
%Define $\tilde{\mathbb{Y}}_k = Y^\T  \tilde{u}_k \in \mathbb{R}^p $, one can write $ \tilde{f}=(\tilde{f}_1^\T, \ldots \tilde{f}_q^\T)^{\T}\in \mathbb{R}^{n\times q} $ with $\tilde{f}_k = L \tilde{\mathbb{Y}}_k = L Y^\T  \tilde{u}_k \in \mathbb{R}^n$.  Then it
%follows that $ \tilde{f}=\tU ^\T YL^\T$. Thus we may write $F - \tilde{F} = Uf - \tU \tilde{f} = U f -  \tU \tU^\T  YL^\T  = Uf -
%UU^\T  YL^\T  + (UU^\T  - \tU \tU^\T  ) Y L^\T $, which gives the decomposition
%\[  F - \tilde{F} = \underbrace{Uf - UU^\T  YL^\T }_{I} + \underbrace{(UU^\T  -\tU \tU^\T  ) Y
%L^\T }_{\text{II}}  \]
As $U^\T U = I_q$, one has
\begin{eqnarray}
\| U \tilde f - UU^\T  YL^\T  \|_F^2  &=& \Tr \{   ( U \tilde f - UU^\T  YL^\T   )^\T ( U \tilde f - UU^\T  YL^\T ) \}  \nonumber
\\
&=& \Tr \{ ( \tilde f - U^\T YL^\T )^\T ( \tilde f - U^\T YL^\T ) \}  = \| \tilde f - U^\T YL^\T  \|_F^2 .  \nonumber
\end{eqnarray}
Furthermore, the rows of $U^\T YL^\T $ are $\tilde{ \hat{ f ^o }}_k  = LY^o_{\cdot k}$. %, where $
%\mathbb{Y}_k = Y^\T u_k=(Y_{i1}, \ldots, Y_{iq})^{\T}$, where $ Y^*_{ik} = f_k(x_i) + \epsilon_{ik} $. 
%{\color{blue} In the theorem, we use $ Y^*_{ik} $, and it looks to me $ k $ and $ k $ have the same meaning. So we may need to unify the notation. Maybe we can change $ k $ to $ k $ in the paper, so in our paper $ k $ always denotes the reduced dimenion $ 1\leq k \leq q $, while $ k $ denotes the original $ p $ dimension, i.e. $ 1\leq k\leq p $. Please check if this is fine.} 
Thus, given the assumptions in Theorem \ref{mainConsistencyTheorem}, we have
%\[ \frac{1}{n} \E (\| I \|_F^2) =  \frac{1}{n} \E  (\| Uf - UU^\T  YL^\T  \|_F^2 )=
%\sum_{k=1}^q \E (\| f_k - L \mathbb{Y}_k \|_n^2) \le q \max_k \E (\| f_k - L
%\mathbb{Y}_k \|_n^2) \le C q n^{-r} .    \]
\begin{eqnarray*}
 \frac{1}{n} \E (\| \text{I} \|_F^2) &=&  \frac{1}{n} \E  ( \| \tilde f - U^\T YL^\T  \|_F^2 )=
\sum_{k=1}^q \E (\| f_k -\hat{ f }^o _k\|_n^2)\\
 &&\le q \max_k \E (\| f_k -\hat{ f }^o_k \|_n^2) \le C q n^{-r} .    
\end{eqnarray*}
On the other hand, as $\| L \| \le C$, one has
\begin{eqnarray*}  \| \text{II} \|_F^2  &=&  \| (UU^\T  - \tU \tU^\T ) YL^\T  \|_F^2 \le C \| UU^\T  - \tU \tU^\T 
\|_F^2 \| Y \|^2 \\
&\le&    q  C \| UU^\T  - \tU \tU^\T  \|^2 \| Y \|^2. 
\end{eqnarray*}
Then taking expectations and using that $ \| UU^\T  - \tU \tU^\T  \| \le 2 \| \sin \{
\Theta(\tU,U) \} \|$ together with an application of Cauchy-Schwarz and lemma
\ref{SingularNormMoment}  gives that
\[   \E  (\| \text{II} \|_F^2)  \le C q \left\{  \left( \frac{p}{n} \right)^2 \wedge 1
\right\}^{1/2} \left\{  \max(p^2 , n^2)  \right\}^{1/2}  \le C q p  \max \left( 1 ,
\frac{p}{n} \right).   \]
As $p=o(n)$,  applying the triangle inequality to the decomposition of $\| \tilde F - \tilde{ \hat{ F } } \|
_F^2$, taking expectations and combining what has been shown together with the fact that $R_n(\hat F) = E \| \tilde F - \tilde{ \hat{ F } } \|
_F^2 / n$ concludes the proof
of the theorem.
\end{proof}

To prove Theorem \ref{mainConsistencyfirststep}, we need to first bound $\E \| \sin \{ \Theta ( \tU , U ) \} \|^4$.  Then one can apply Theorem \ref{mainConsistencyTheorem} to reach the final conclusion of Theorem \ref{mainConsistencyfirststep}. 
\begin{proof} [Proof of Theorem \ref{mainConsistencyfirststep}]
%{\color{blue} (change the notations $ \tilde{F} $, $ \tilde{f} $, $\tilde{U}$ accordingly.)}
As noted above it is enough to show that with $\tilde F = U \tilde f$ and $Y = \tilde F + Z$, %in both the fixed and random design cases specified in the theorem, 
the left singular vectors of the singular value decomposition of $Y$, $Y = \tU \hat \Sigma \hat V ^\T $, satisfy the bound
\[   \E [  \| \sin \{ \Theta ( \tU , U ) \} \|^4 ] \le C  \left(  \frac{p}{n} \right)^2
\wedge 1.  \]
As the $\tilde f_k$'s are not
necessarily orthogonal in $\mathbb{R}^n$, we need to consider that the singular value decomposition of $\tilde F$ is of the form $\tilde F = P
\Gamma Q^\T $, with $P$ possibly spanning a different subspace from $U$.  Let $\mathcal{A}$ be the event given by
\[  \mathcal{A} = \Big\{ \Sp(P) = \Sp(U) \Big\}.  \]
Then Lemma \ref{SampledSpan} shows that %in both cases under consideration 
$\P(\mathcal{A}) \ge 1- B / n^2$.
Then given that $\| \sin \{ \Theta (P,U) \} \| \le 1$ for any pair of subspaces and on the event $\mathcal{A}$, $\sin \{ \Theta(P, U) \} = 0$, it holds that
\begin{eqnarray*}   \E [ \| \sin \{ \Theta ( P , U ) \} \|^4 ] &=& \E [ \| \sin \{ \Theta ( P , U ) \} \|^4 \ind_\mathcal{A} ] +  \E [
\| \sin \{ \Theta ( P , U ) \} \|^4 \ind_{\mathcal{A}^c} ] \\
&\le& \P (\mathcal{A}^c) \le B/n^2 .   
\end{eqnarray*}
Thus we may conclude that
\[   \E [ \| \sin \{ \Theta ( P , U ) \} \|^4 ]\le  C \left(  \frac{p}{n} \right)^2 \wedge 1.
\]
By the triangle inequality, 
\begin{eqnarray*}   
\| \sin \{ \Theta ( \tU , U ) \} \|^4 &\le& \left[  \| \sin \{ \Theta ( P , U ) \} \| +  \|
\sin \{ \Theta ( P, \tU ) \} \| \right]^4 \\
&\le& C \left[  \| \sin \{ \Theta ( P , U ) \} \|^4 +
\| \sin \{ \Theta ( P,\tU ) \} \|^4 \right]. 
\end{eqnarray*}
To complete the proof, it suffices to show that $\E [ \| \sin \{ \Theta ( P,\tU ) \} \|^4]$ satisfies the
bound of the theorem.  
Now we consider the event $\mathcal{E}$ that the singular values of $\tilde F$ scale like $n^{1/2}$,
\[ \mathcal{E} = \Big\{ n^{1/2} c \le  \Gamma_{qq} \le n^{1/2} C \Big\},  \]
where $c, C > 0$ satisfy $c \le \min_k \|f_k\|_{L^2} \le \max_k \| f_k \|_{L^2} \le C$ and $\Gamma_{qq} $ is the $q$-th element of $\Gamma$.  Then Lemma \ref{QuantitativeSubspaceAngles} implies that $\P(\mathcal{E}) \ge 1 - D / n^2$.  Denote $ \E \left(  \cdot | \mathcal{D} \right)$ the expectation conditioned on design. As $\| \sin \{ \Theta (P, \tU ) \} \| \le 1$, we may decompose $ \E  \| \sin \{ \Theta (P, \tU  ) \}\|^4 $ as
\begin{eqnarray*}   \E  [\| \sin \{ \Theta (P, \tU  ) \}\|^4  ]  &=&  \E  \left( \E \left[ \left. \| \sin \{ \Theta (P,
\tU  ) \} \|^4 \right| \mathcal{D} \right] \right) \\
&\le& \P(\mathcal{E}^c) +\E  \left( \E \left[ \left. \| \sin \{ \Theta (P,
\tU  ) \} \|^4 \right| \mathcal{D} \right]  \ind_\mathcal{E} \right) .  
\end{eqnarray*}
We extend the proof of theorem 3 in \cite{cai2016rate} to bounds for fourth moment of $\| \sin \{ \Theta (P, \tU  ) \}\|$ such that
\[  \E \left[ \left. \| \sin \{ \Theta (P, \tU  ) \}\|^4 \right| \mathcal{D} \right]  \le
C \left[  \frac{ p \{ \sigma_q^2( \tilde F ) + n \}  }{ \sigma_q^4 ( \tilde F )  } \right]^2 \wedge 1.
\]
On the event $\mathcal{E}$.  By construction, on the event $\mathcal{E}$ we also have that $ n^{1/2} c \le \Gamma_{qq} = \sigma_q( F) \le n^{1/2} C$ and so

\[    \frac{ p \{ \sigma_q^2(  \tilde F ) + n \}  }{ \sigma_q^4 (  \tilde F )  }  \le \frac{ p(C + 1)n }
{c^2 n^2} \le C  \frac{p}{n}.   \]
Using the bound on $\P(\mathcal{E}^c)$ and piecing together what has been shown concludes the proof
of the theorem.
\end{proof}

%{\color{blue} (Is the extension useful in proving any of the lemmas and theorems? If not, what is the motivation to have this extension?)}. 
%{\color{red} (Yes, it is what guarantees that the bound referenced and used in the proof above works for 4th moments.)}. 
Here we carefully go through the arguments of Theorem 3 in
\cite{cai2016rate} to guarantee that they hold in our case, using the notation of
that paper.  As noted in that proof, by
symmetry it is enough to extend the method of proof for the right singular vectors and
we employ the same concentration results outlined at the outset of the proof.  This
is because the left singular vectors of $Y$ are just the right singular vectors of $Y^\T$.
Thus we can apply the results for the right singular vectors of $Y$ to $Y^\T$ to get
bounds for estimation of the left singular vectors of $Y$.
%This explains why $p_1$ and $p_2$ are flipped in the results).
Based on the proof we just need to extend the inequality to the case that $
\sigma^2_r(X) \ge C_{\text{gap}} \{ (p_1p_2)^{1/2} + p_2 \}$.
Now, under the event $\mathcal{Q}$ given there, the inequality given there implies that
\[    \| \sin \Theta(\hat{V},V) \|^4 \le C \left\{ \frac{  \sigma^2_r(X) + p_1  }{
\sigma^4_r(X)  }  \right\}^2 \| \mathbb{P}_{YV} Y V_{\perp} \|^4 ,   \]
which, in turn, gives that
\[   \E  \left\{ \| \sin \Theta(\hat{V},V) \|^4 \right\}  \le  \P(\mathcal{Q}^c) + C \left\{ \frac{  \sigma^2_r(X)
+ p_1  }{ \sigma^4_r(X)  }  \right\}^2  \E ( \| \mathbb{P}_{YU} Y U_{\perp} \|^4 \ind_{\mathcal{Q}}).
\]
In this regime (where $ \sigma^2_r(X) \ge C_{\text{gap}} \{ (p_1p_2)^{1/2} + p_2 \}$) a
bound is given for $\P(\mathcal{Q}^c)$ of
\[   \P(\mathcal{Q}^c)  \le \exp \left\{ - c \frac{  \sigma^4_r(X)}{  \sigma^2_r(X) + p_1} \right\}
.  \]
When the fraction in the exponent diverges, the exponent tends to zero faster than any
polynomial and so we have
\[   \P(\mathcal{Q}^c)  \le   C \left\{ \frac{  \sigma^2_r(X) + p_1  }{ \sigma^4_r(X)  }  \right\}^2
\le   C \left\{ \frac{ p_2(  \sigma^2_r(X) + p_1 ) }{ \sigma^4_r(X)  }  \right\}^2 . \]
Considering what happens if the exponent is not diverging, we see that this holds in
either case.  Thus for the desired extension, it remains to show that 
$\E ( \| \mathbb{P}_{YU} Y U_{\perp} \|^4 \ind_{\mathcal{Q}} )$ $\le C p_2^2$.   As in the proof, we let
$T = \| \mathbb{P}_{YU} Y U_{\perp} \|$ and note that (using the concentration bounds
employed in the theorem)
\be
\E (T^4 \ind_\mathcal{Q})  & \le &  \E \{  T^4 \ind_{ (  T^2 \le \sigma_r^2(x) + p_1  ) } \} =
\int_0^{\infty} \P \left\{ T^4 \ind_{ (  T^2 \le \sigma_r^2(x) + p_1  ) } > t \right\}
dt  \nonumber  \\
&\le & \delta p_2^2 + \int_{\delta p_2^2}^{ \{\sigma_r^2(x) + p_1\}^2 } \P \left( T >
t^{1/4} \right) dt   \nonumber \\
&\le & \delta p_2^2 + C \int_{\delta p_2^2}^{ \{\sigma_r^2(x) + p_1\}^2 }  \left(
\exp(Cp_2 - ct) + \exp[-c\{\sigma_r^2(X) + p_1\}]     \right) dt      \nonumber \\
&\le& \delta p_2^2  +   C  \{\sigma_r^2(x) + p_1\}^2   \exp[-c \{ \sigma_r^2(X) + p_1 \} ] + \nonumber\\ 
&& + C
\exp(Cp_2) \cdot \exp(-c \delta p_2) / c.  \nonumber
\ee
From this, we see that choosing $\delta$ large enough we may guarantee that $\E (T^4
\ind_\mathcal{Q})  \le C p_2^2$, which is what we wanted.  This is the final piece needed in
showing that the form of the bound we wanted holds for the fourth moment.

\section{Relevant Lemmas for Main Theorems}\label{lemmasMainTheorem}
We introduce the auxiliary lemmas for main theorems in Section \ref{Appendixprooftheorem}, the proofs of which are deferred to the Section \ref{prooflemmas} of the supplementary file.

The next two results bound fourth moments of $\| Y \|$ and $\sin \{ \Theta(\tilde{U},U) \} $. %in
%the fixed and random design cases on $[0,1]$ as illustrative examples. 
In the results
that follow, $Z \in \mathbb{R}^{p \times n}$ is composed of independent and identically distributed $\mathcal{N}(0,1)$ entries. % {\color{blue} (As I commented in the main paper, we may need to change $ Z $ to other symbols. In addition, we may be careful about the dimension of $ Z $, i.e. whether it is $ p\times n $ or $ n\times p $ may also affect the definition of $ \epsilon_{ij}$ in the main paper. As Biometrika does not allow bold font, we need to be careful about the notation.)}
%{\color{red} Not sure how we want to proceed here as this all relies on definitions of the initial versions of the paper.}
The first main lemma is as follows
\begin{lem}
\label{SingularNormMoment}
With $Y = \tilde F + Z \in \mathbb{R}^{p \times n}$ denoting the data matrix and $\| Y \|
= \max_i \sigma_i (Y)$ the operator norm, or maximum singular value, of $Y$ we
have that
\[    \E (\| Y \|^4 ) \le C \max(p^2,n^2)   \]
holds 
%in both the fixed and random design cases 
when $\max_k \|f_k \|_{\infty} \le B$.
\end{lem}

Next, we show the lemma needed in the proof of Lemma \ref{SingularNormMoment}.  
\begin{lem}
\label{MomentBd}
If $X \ge 0$ is a positive random variable and for $a,b > 0$ we have $
\P( X > a + bt) \le 2 \exp(- t^2)$ for all $t \ge 0$ then it follows that $\E(X^4) \le C \max(a^4, b^4)$.
\end{lem}

The next lemma quantifies the discrepancy between $\langle \cdot , \cdot \rangle_n$
and $\langle \cdot, \cdot \rangle_{L^2}$.
%in fixed and random design cases on $[0,1]$.
These are crucial to the proofs of Lemmas \ref{SampledSpan} and \ref{QuantitativeSubspaceAngles} which, in turn are crucial to the proofs of the main theorems of the paper. 
%{\color{blue} This sentence is vague. Can you illustrate it more clearly where we use the lemmas?}
%{\color{red} Have quoted the lemmas that rely on these and there is a description of why this is important below (immediately preceding lemmas 6 and 7, the important/dependent lemmas just referenced.)}
\begin{lem}
\label{rndDes}
Suppose that $x_i \sim \mathcal{U} [0,1]^d$ are independently drawn from the uniform distribution
on the unit cube in $\Rd$ and $f_1,\dots,f_q$ are bounded and orthogonal in $L^2 [0,1]^d
$, satisfying $\max_{i \le q} \| f_i \|_{\infty} \le B$.   Then,
\[    \P(  \max_{i,j} | \langle  f_i , f_j \rangle_n  - \langle  f_i , f_j \rangle |  >
2 \delta B^2  )   \le  q(q + 1)  \exp \left(   -  \frac{ n \delta^2   }{    2    }
\right)    \]
and so, with probability $\ge 1 - q(q + 1)/n^2$,
\[       \max_{i,j}  | \langle  f_i , f_j \rangle_n  - \langle  f_i , f_j \rangle |  \le
4 B^2  \left( \frac{\log n}{n} \right)^{1/2}  .    \]

\end{lem}

%In the fixed design case, one may prove the following theorem, as in \cite{Chazelle2000}.
%\begin{lem} 
%\label{Chazelle}
%(Chazelle) Suppose that $f$ is of bounded variation, so that
%\[  V(f) = \int_{[0, 1]} |f'| <  \infty. \]
%Then for any collection of points $0 \le x_1 \le x_2 \le \dots \le x_n \le 1$, with $H_n(x) = |\{ i: x_i \le x \}| / n$ being the cumulative density function of the design points $x_i$, we have that
%\[ \left| \frac{1}{n} \sum_{i=1}^n f(x_i) -  \int_0^1 f(x) dx \right| \le  V(f) \|H_n - x \|_{\infty}. \]
%\end{lem}

%This leads to the following lemma for the discrepancies between the $\langle \cdot , \cdot \rangle_n$
%and $\langle \cdot, \cdot \rangle_{L^2}$ in the fixed design case.
%\begin{lem}
%\label{fxDes}
%Suppose that the design points $x_i$ are \emph{near equi-spaced}, so that the empirical cumulative density function of the $x_i$, $H_n$, satisfies $n \| H_n - x \|_{\infty} = O(1)$ and suppose further that the $f_k$'s are both bounded and of bounded variation, so that $\max_{i \le q} \| f_i \|_{\infty} \le B_1$ and $ \max_{i \le q} \| f_i'   \|_{L^1} \le B_2 $.  Then it follows that
% \[   \max_{i,j}  | \langle  f_i , f_j \rangle_n  - \langle  f_i , f_j \rangle |  \le
% C/n . \]
%\end{lem}

Let $\tilde{F} = U \tilde f \in \mathbb{R}^{p \times n}$, as defined above, represent the sampled version of the singular value decomposition representation of the target
\[  F = \sum_{k=1}^q u_k \otimes f_k = \sum_{k=1}^q \sigma_k u_k \otimes v_k ,   \]
where we have set $\sigma_k = \|f_k\|_{L^2}$ and $v_k = f_k / \|f_k\|_{L^2}$.   Further, we let $\tilde{F} = P \Gamma Q^\T $ denote the singular value decomposition of $\tilde{F}$.   Recall that the matrix $\tilde f \in \mathbb{R}^{q \times n}$ collects the sampled values of the $f_k$ in its rows.  

Due to sampling, it is not clear whether $\tilde F$ is a close approximation to $F$ in either of the following senses:

1, The matrix $\tilde F$ provides a close approximation to the singular vectors we wish to estimate in that the span's are the same, i.e. $\Sp( \tilde F) = \Sp(P) = \Sp(U)$.  This in turn guarantees that $\sin \Theta(P, U) = 0$.

2, The matrix $\tilde F$ is ``large'' enough to separate signal from noise in estimating the $U$. 

The following lemmas resolve these issues and are central to the proof of Theorem \ref{mainConsistencyfirststep}.

\begin{lem}
\label{SampledSpan}
%In both the fixed and random design cases, 
We eventually have  $\Sp(\tilde F) = \Sp(U)$ with probability greater than or equal to $1- B/n^2$ for some fixed $B \ge 0$.  As $\Sp(\tilde F)  = \Sp(P)$, with $P$ from the singular value decomposition of $\tilde F = P \Gamma Q^\T $, this guarantees $\Sp(P) = \Sp(U)$ and hence $\sin \Theta(P, U) = 0$.  
\end{lem}

\begin{lem}
\label{QuantitativeSubspaceAngles}
Let $\tilde{F}$ represent the sampled version of the singular value decomposition representation of the target
\[  F = \sum_{k=1}^q u_k \otimes f_k = \sum_{k=1}^q \sigma_k u_k \otimes v_k  ,   \]
where $ \otimes $ is the Kronecker product.  Under 
%either of the fixed or random design 
the conditions of Theorem \ref{mainConsistencyfirststep}, where we have set $\sigma_k = \|f_k\|_{L^2}$ and $v_k = f_k / \|f_k\|_{L^2}$, if $\gamma$ is one of the top $q$ singular value of $\tilde{F}
$, $\gamma$ satisfies
\[  \min_{i \le q} | \gamma - n^{1/2} \sigma_i | \le Cq^{1/2} \left( n \log n \right)^{1/4}  \]
with probability at least $1-B/n^2$ for some $B \ge 0$. In particular, it follows that if $c \le
\min_k \sigma_k \le  \max_k \sigma_k \le C$ then, with possibly adjusted constants, $
\gamma$ satisfies
\[   n^{1/2}c \le \gamma \le n^{1/2}C \]
with probability at least $1-B/n^2$, for some $B \ge 0$ and large enough $n$.
\end{lem}

For the proof of Lemma \ref{QuantitativeSubspaceAngles}, we need a well known perturbation result for matrices \citep{weyl1912asymptotische}, which will ease the proof of this result considerably.

\begin{lem}
\label{WeylThm}
(Weyl) Let the eigenvalues of real symmetric matrices $A$ and $A + E$ be $\lambda_1 \ge \lambda_2 \ge \dots \ge \lambda_n \ge 0$ and $\tilde \lambda_1 \ge \tilde \lambda_2 \ge \dots \ge \tilde \lambda_n \ge 0$ respectively. Then $\max_i |\lambda_i - \tilde \lambda_i| \le \| E \|_2$.  
\end{lem}

\section{Details on Examples and Attained Rates}\label{lemmasMainExamples}
We introduce the lemmas and theorems for examples in the Section \ref{Theoretical guarantees}, the proofs of which are deferred to the Section \ref{prooflemmas} of the supplementary file.

\subsection{Local Polynomial Regression}\label{local}
For this section, the smoothness class of primary concern is the \emph{H\"older class},  $\Sigma(\beta, L)$.  For any real number $x$, let $\lfloor x \rfloor$ represent the largest integer strictly less than $x$.  Then $\Sigma(\beta, L)$ consists of all functions $f$ which are $l = \lfloor \beta \rfloor$ times differentiable and whose $l$th derivative $f^{(l)}$ satisfies
\[  |f^{(l)}(x) - f^{(l)}(y)| \le L | x - y |^{\beta - l},  \]
for all $x, y$ in the domain of interest.  

It is well known that in the fixed design case, where we roughly have $x_i = i/n$, if the kernel and bandwidth are properly chosen, then local polynomial smoothing gives an estimator $\hat{f}$ of $f$ from the data
\[ y_i = f(x_i) + z_i, \]
which satisfies
\[  \E_f \| \hat{f}_n - f \|_n^2  \le C n^{- 2\beta / (1 + 2\beta)}.  \]
Furthermore, this rate is minimax optimal.  For proof and in depth setup, see proposition 1.13 and theorem 1.6 in \cite{Tsybakov:2008:INE:1522486}.   In the random design case, there don't seem to be any results on convergence in the metric we want, namely $\E_f \| \cdot \|_n^2$.  

One remedy to this is to adopt a similar approach to that in \cite{Cai1999RUD}, and slightly modify the local polynomial regression strategy.  To this end, let $0 \le x_{(1)} \le \dots \le x_{(n)} \le 1$ represent the order statistics of the uniform design and relabel the $y_i$'s and $z_i's$ according to these so that $y_i = f(x_{(i)} ) + z_i $ generate the observations.  In the recovery procedure, we pretend that $x_{(i)}$ is $\delta_i  =  \E x_{(i)} = i / (n + 1)$ so that we perform local polynomial regression as if the observations were $ \left( \delta_i ,  y_i \right) $ for $ i=1,\dots,n $. That is, with $K$ being a kernel satisfying the right conditions and
\[  U(u) = \left(  1, u, \frac{u^2}{2} , \dots, \frac{u^l}{l!}\right)^{\T},  \]
we form the estimate
\[  \hat{\theta}(x) = \arg \min_{\theta}  \sum_{i = 1}^n \left\{ y_i   -  \theta^{\T} U \left( \frac{\delta_i - x}{h} \right)  \right\}^2 K \left( \frac{\delta_i - x}{h} \right)   \]
and we estimate $f$ by
\begin{equation}
\label{LPSestimate} 
\hat{f}(x) =  U(0)^{\T} \hat{\theta}(x). 
\end{equation}
Then for a given $x$, $\hat{f}(x)$ is linear in the $y_i$ in that one may show
\[ \hat{f}(x)  = \sum_{i=1}^n W_{n, i}(x) y_i ,  \]
as for a standard local polynomial estimator.  Further, the $W_{n, i}(x)$ are now completely deterministic, satisfying all of the properties derived in \cite{Tsybakov:2008:INE:1522486}.  %This is nice, because we can show that the bias and variance (conditioned on design) satisfy similar properties to those derived in \cite{Tsybakov:2008:INE:1522486} for the fixed design case.  
We can then show that this estimator achieves the rate we want in the metric we need it to, as the following theorem guarantees. 

%{\color{blue} (this sentence is a bit vague. Maybe we can specify the rate.)}
%{\color{red} (Hopefully reference to the theorem, which is a bit more precise, eliminates the ambiguity.)}
%{\color{blue} We need to list the theorem as Theorem 3 as we already have two theorems in the main paper. }
%{\color{red} (Am not sure how to cross reference theorems across separate latex files; maybe someone else in the group knows.  Otherwise, I can help figure something out.)}

\begin{thm}\label{LPStheorem} (Local Polynomial Smoothing)
Suppose that $f$ belongs to the H\"older class $\Sigma(\beta, L)$, the design is uniform random and the kernel $K$ satisfies the properties outlined in \cite{Tsybakov:2008:INE:1522486}.  Then we may be assured that $\hat{f}$ outlined above satisfies 
\[  \E_f \| \hat{f} - f \|_n^2  \le C n^{- 2\beta / (1 + 2\beta)}.  \]
\end{thm}

For local polynomial smoothing, we have $\|L\| \le C$.   This follows from a result for bounds of eigenvalues of matrices.  Let $A = (a_{kl})_{k,l=1}^n$ be an $n \times n$ matrix and set
\[ R_k = \sum_l | a_{kl} | \hspace{5pt} \text{ and } \hspace{5pt} C_l = \sum_k | a_{kl} | . \]
Then one can show that the eigenvalues of $A$, $\mu(A)$, are bounded by
\[  \mu(A) \le \min \left( \max_k R_k, \max_l C_l \right) \le \max_k R_k . \]
In the case of local polynomial smoothing, the $L$ satisfies $L_{ij} = W_{n, j}(x_i)$ and from \cite{Tsybakov:2008:INE:1522486} we know that
\[ R_i = \sum_j | L_{ij} | = \sum_j | W_{n, j}(x_i) | \le C .  \]
Thus we have that  $\mu(L) \le C$ and thus $\| L \| \le C$.

\subsection{Truncated Series Estimation}\label{truncated}
For extensive setup and analysis of fixed design for Fourier basis and Sobolev smoothness, see \cite{Tsybakov:2008:INE:1522486}.   Here the smoothness class of interest is the Sobolev class of periodic functions of integer smoothness $\beta$, denoted by $W^p(\beta, L)$. To define this class of functions, we start with the Sobolev class $W(\beta, L)$ defined by
\[ W(\beta, L) = \left\{ f \in L^2[0, 1] : f^{(\beta - 1)} \in \mathcal{C}[0, 1] \text{ and } \int_0^1 (f^{(\beta)})^2 \le L \right\} , \]
where $\mathcal{C}[0, 1]$ is the collection of absolutely continuous functions on $[0, 1]$.  The function class of interest, $W^p(\beta, L)$, is then defined by
\[ W^p(\beta, L)  =  \left\{ f \in  W(\beta, L): f^{(j)} (0) = f^{(j)}(1) \text{ for } j=0, 1, \dots, \beta - 1 \right\} .  \]
Fix the Fourier basis, where $\varphi_1 = 1$ and $\varphi_{2k} = 2^{1/2} \cos(2 \pi x)$, $\varphi_{2k + 1} = {2}^{1/2} \sin(2 \pi x)$ for $k \ge 1$.  It is known that every $f  \in W^p(\beta, L)$ has a Fourier expansion of the form
\[ f = \sum_{k=1}^{\infty} f_k \varphi_k \]
and that the coefficients of all $f  \in W^p(\beta, L)$ lie in an ellipsoid of the form
\[ Q(\beta, C) = \left\{ (c_k) \in \ell_2 : \sum_{k \ge 1} a_k^2 c_k^2 \le C \right\} , \]
where the $a_k \sim k^{\beta}$.

Now let $\{ \varphi_k \}_{k \ge 1}$ be an orthonormal basis of $L^2[0, 1]$, $x_i \stackrel{i.i.d.}{\sim} U[0, 1]$ represent a uniform random design on $[0, 1]$ and set
\[  \hat{c}_k  = \frac{1}{n} \sum_{i = 1}^n y_i \varphi_k(x_i)  \hspace{5pt} \text{ and } \hspace{5pt} \hat{f}_{n, K} = \sum_{k = 1}^K \hat{c}_k \varphi_k .   \]
For analysis of the estimator $\hat{f}_{n, K} $, we set 
\[ \overline{c}_k = \frac{1}{n} \sum_{i = 1}^n f(x_i) \varphi_k(x_i)      \]
and notice that we have $\E_{f, \mathbb{X}}  (\hat{c}_k) = \overline{c}_k $ and $\E_{f}  (\overline{c}_k ) =  \E_{f} ( \hat{c}_k ) = c_k .$
Setting $f_K = \sum_{k > K} c_k \varphi_k$, we may establish the following theorem
\begin{thm} (Oracle inequality for truncated series estimation)
\label{OrrTES}
The estimator $\hat{f}_{n, K}$ defined above satisfies the risk bound
\[ \E_f (  \| f - \hat{f}_{n, K}  \|_n^2 ) \le 2 \frac{ (1 + K / n ) (1 + \|f\|_{L^2}^2 )  K }{ n }  +  2 \E_f (\| f_K \|_n^2).    \]
\end{thm}

Now  assumptions on the decay of the $c_k$ will provide bounds on the error of the estimator.  For instance, assuming that $f \in W^p(\beta, L)$ implies that $(c_k) \in Q(\beta, C)$, as above.  This implies that
\bea
\E_f (\| f_K \|_n^2) &=& \| f_K \|_{L^2}^2 = \sum_{k > K} c_k^2 \nonumber  \\
&=& \sum_{k > K} a_k^{-2} a_k^2 c_k^2  \le a_K^{-2}\sum_{k > K} a_k^2 c_k^2 \le C a_K^{-2} \le C K^{-2\beta}.
\eea
Employing the Theorem \ref{OrrTES}, this gives that 
\[ \E_f  ( \| f - \hat{f}_{n, K}  \|_n^2)  \lesssim \frac{K}{n} + K^{-2\beta}.   \]
Choosing $K \sim n^{1 / (1 + 2\beta)}$  gives  $\E_f   (\| f - \hat{f}_{n, K}  \|_n^2)  \lesssim  n^{-2\beta / (1 + 2\beta)}$ which is the known minimax rate for these classes.

For projection estimation, we have $\|L\| \le C$.  This follows since the least squares estimator is a projection estimator and hence all eigenvalues are less or equal to $1$.

\subsection{Reproducing Kernel Hilbert space Regression}\label{RKHS}

Let $K$ be a positive semidefinite kernel function $[0, 1]^2$ and $\mathcal{H} = \mathcal{H} (K)$ the associated reproducing kernel Hilbert space on $[0, 1]$ with norm $ \| \cdot \|_{ \mathcal{H}} $.  We are interested in the performance of penalized estimation strategies
\[ \hat{f}_{n, K} = \arg  \min_{g \in \mathcal{H}} \left\{ \frac{1}{2}  \|y - g\|_n^2 + \lambda_n \| g \|_{\mathcal{H}}^2 \right\},  \]
which, by the representer theorem,  takes the form of linear smoothers.  In particular, if we set $K = (K(x_i, x_j)/n)_{i, j = 1}^n$ then we find that
\[ \hat{f}_{n, K} = n^{-1/2} \sum_{i=1}^n \hat{\alpha}_i K(\cdot, x_i)  \hspace{5pt} \text{ where }  \hspace{5pt} \hat{\alpha} = n^{-1/2}(K  + \lambda I )^{-1} y  \]
so that
\[ ( \hat{f}_{n, K} (x_1) , \dots, \hat{f}_{n, K} (x_1) )^{\T}  = K (K  + \lambda I )^{-1} y  = Ly  \]
is linear in the data, with $L = K (K  + \lambda I )^{-1}$; here we naturally have $\| L \| \le 1$, as can be seen expanding $L$ in the eigendecomposition of $K$.  See chapter 12 of \cite{Wainwright2019} for details and more extensive development.

In analyzing the estimator $\hat{f}_{n, K}$, the difference symmetrized space, $\partial \mathcal{H}$, defined by
\[ \partial \mathcal{H} = \mathcal{H}  - \mathcal{H}  =  \left\{ \left. g - h \right| g, h \in \mathcal{H}   \right\}, \]
turns out to be an important quantity.  For $\delta > 0$ we define the localized sets
\[ \partial \mathcal{H}_n(\delta) =\left. \left\{ g \in \partial \mathcal{H}  \right|  \| g \|_{ \mathcal{H}}  \le 3, \|g\|_n \le \delta \right\}, \]
consisting of functions in $\mathcal{H}$ with small empirical norm.  The localized gaussian complexity,  $\mathcal{G}_n(\delta)$,  associated with $\mathcal{H}$ and a given set of sampling points $\mathbb{X}_n = \{x_1, \dots, x_n\}$, is defined by
\[  \mathcal{G}_n(\delta)  =  \E_w \left\{ \left. \sup_{   g  \in  \partial \mathcal{H}_n(\delta)  }  \left|  \langle w, g \rangle_n \right|  \right| \mathbb{X}_n \right\} , \]
where $w \sim N_n(0, I_n) $, and can be used to control the error of the estimation procedure $\hat{f}_{n, K}$.  For this setup, the gaussian complexity is known to satisfy the bound
\[    \mathcal{G}_n(\delta)  \le   \mathcal{G}^K_n(\delta) = \left\{ \frac{2}{n} \sum_{k=1}^{n} \min( \delta^2 , \hat{\mu}_j)\right\}^{1/2},  \]
with $ \hat{\mu}_j$ being the eigenvalues of the empirical kernel matrix $K$.   Now with $\delta_n > 0$ being a positive solution of $2 \mathcal{G}_n( \delta)  \le R \delta^2$, and $\lambda$ chosen from the range  $[2 \delta_n^2, C \delta_n^2]$, $C > 2$,  as shown in chapter 13 of \cite{Wainwright2019}, the estimator $\hat{f}_{n, K}$ satisfies the the oracle inequality
\[ \E ( \| f - \hat{f}_{n, K}  \|_n^2 \left| \mathbb{X}_n \right. ) \le C \left\{ \inf_{ \| g \|_{ \mathcal{H}} \le R }  \|f - g\|_n^2 + R^2  \delta_n^2  \right\}  .  \]
Thus if we choose $\delta_n > 0$ being the smallest number satisfying $2 \mathcal{G}^K_n( \delta)  \le R \delta^2$ and $f \in \mathcal{H}$ with $\| f \|_{ \mathcal{H}} \le R$, then we have that $ \E ( \| f - \hat{f}_{n, K}  \|_n^2 \left| \mathbb{X}_n \right. )  \le C  \delta_n^2$ and so $ \E ( \| f - \hat{f}_{n, K}  \|_n^2 )   \le C  \E_{\mathbb{X}_n} \delta_n^2$, which we will proceed to bound for some concrete examples, using known eigenvalue decay of some common operators.   

\subsubsection{Oracle inequality applied to random design}
As per the program outlined above, the goal is to bound the estimation error $ \E( \| f - \hat{f}_{n, K}  \|_n^2)$ by bounding $\E_{\mathbb{X}_n} \delta_n^2$ %under %fixed and random design 
for various $\mathcal{H} = \mathcal{H} (K)$.   This is done by relating the eigenvalues $ \hat{\mu}_j$ of the empirical kernel matrix $K$ to those $\mu_j$ of the underlying kernel $K$, viewed as an integral operator; the order of these eigenvalues are known for various important $\mathcal{H} = \mathcal{H} (K)$, which allows us to bound the rate of estimation.  In particular, if we know the order of the $\mu_j$, we can calculate the minimum positive solution to the population complexity equation $2 \overline{\mathcal{G}}_n^K( \delta)  \le R \delta^2$, say $\gamma_n$, where
\[ \overline{\mathcal{G}}_n^K( \delta)  =  \left\{  \frac{2}{n} \sum_{j=1}^{\infty} \min(\mu_j, \delta^2) \right\}^{1/2} .  \]
If the $ \hat{\mu}_j$ are not overly different from the $\mu_j$, the hope is that $\delta_n$ and $\gamma_n$ are not overly different.  %Showing this is approached differently for fixed and random design.  

First we need a couple of new definitions.  We need the notion of \emph{local Rademacher complexity}, $\mathcal{R}_n(\delta)$, of a function class defined by
\[ \mathcal{R}_n(\delta) =  \E_b \left\{ \left. \sup_{   g  \in  \partial \mathcal{H}_n(\delta)  }  \left|  \langle b, g \rangle_n \right|  \right| \mathbb{X}_n \right\} , \]
where $b \in \{-1, 1\}^n$ is a collection of independent Bernoulli variables.  As with the local Gaussian complexity, $\mathcal{G}_n(\delta)$, this quantity is random, depending on the design $\mathbb{X}_n$.   Similarly, we may define the corresponding population quantities $\overline{\mathcal{R}}_n(\delta)$ and $\overline{\mathcal{G}}_n(\delta)$  
%{\color{red} (This is not quite true, as the population quantities take $\sup$'s over $\mathcal{H}(\delta)$; clean this up!) },
\[ \overline{\mathcal{R}}_n(\delta) = \E_{b, \mathbb{X}_n} \left\{  \sup_{   g  \in  \partial \mathcal{H}(\delta)  }  \langle b, g \rangle_n   \right\}   \hspace{5pt} \text{ and }  \hspace{5pt}  \overline{\mathcal{G}}_n(\delta) =\E_{w, \mathbb{X}_n} \left\{  \sup_{   g  \in  \partial \mathcal{H}(\delta)  }  \langle w, g \rangle_n   \right\} , \] 
%\[ \overline{\mathcal{R}}_n(\delta) = \E_{\mathbb{X}_n} \mathcal{R}_n(\delta)   \hspace{5pt} \text{ and }  \hspace{5pt}  \overline{\mathcal{G}}_n(\delta) = \E_{\mathbb{X}_n} \mathcal{G}_n(\delta), \] 
which average over design and take $\sup$'s over $\mathcal{H}(\delta)$.   It is known that $\overline{\mathcal{G}}_n^K(\eta)$ is an upper bound for the corresponding population quantities so that $\overline{\mathcal{G}}_n(\delta), \overline{\mathcal{R}}_n(\delta) \le \overline{\mathcal{G}}_n^K(\delta)$, which will be crucial to what follows.  As above, we let $\gamma_n$ be the smallest positive solution of $\mathcal{G}_n^K(\delta) \le \delta^2$ and notice that by the bound above, we know that $\xi_n$ dominates the smallest positive solutions to $\overline{\mathcal{G}}_n(\delta), \overline{\mathcal{R}}_n(\delta) \le \delta^2$.  As $g \in \mathcal{H}(\delta), \mathcal{H}_n(\delta)$ implies $-g \in \mathcal{H}(\delta), \mathcal{H}_n(\delta)$, it follows that the absolute value in the definition of $\mathcal{G}_n(\delta)$ is redundant so that
\[ \mathcal{G}_n(\delta) =  \E_w \left\{ \left. \sup_{   g  \in  \partial \mathcal{H}_n(\delta)  }   \langle w, g \rangle_n   \right| \mathbb{X}_n \right\}.  \]
The same observation can be made about the local Rademacher complexity.   This eases the development of concentration inequalities for these quantities in the development that follows.

Our aim is to relate the solution of the empirical gaussian complexity to the solution to the population gaussian complexity, $\overline{\mathcal{G}}_n(\delta)$, which we can calculate a bound for (via $\mathcal{G}_n^K(\delta)$),
\[ \overline{\mathcal{G}}_n(\delta) = \E_{x, g} \sup_{f \in \partial \mathcal{H}(\delta)}  \langle g , f \rangle_n.  \]
The proof can be divided into two steps. We begin by introducing another version of the empirical complexity
\[ \widehat{\mathcal{G}}_n(\delta) = \E_{g} \left\{ \sup_{f \in \partial \mathcal{H}(\delta)}  \langle g , f \rangle_n \right\}  \hspace{5pt} \text{ where }  \hspace{5pt}  \partial \mathcal{H}(\delta) =\left. \left\{ g \in \partial \mathcal{H}  \right|  \| g \|_{ \mathcal{H}}  \le 3, \|g\|_2 \le \delta \right\}, \]
which is random in the design.  We show that this is a self-bounding function and that it concentrates swiftly at its expectation,  $\overline{\mathcal{G}}_n(\delta) = \E_x \left\{ \widehat{\mathcal{G}}_n(\delta) \right\}$.  At the same time, we show that on sets of high probability $\mathcal{G}_n(\delta) \approx \widehat{\mathcal{G}}_n(\delta)$.   This allows us to show that the empirical $\delta_n$ is close to the population $\overline{\delta}_n$, for which we can calculate a bound, in probability and expectation.

In this direction, we develop a crucial lemma.  First, with $\gamma_n$ as above and for $\lambda \ge 1$, we define two sets $\mathcal{E}_0(\lambda)$ and $\mathcal{E}_1(\lambda)$ by:
\begin{eqnarray*} 
\mathcal{E}_0(\lambda) &=& \left\{  {x} \left| \sup_{f \in \mathcal{F}} \frac{ \left|  \|f\|_n^2 - \|f\|_2^2  \right| }{ \|f\|_2^2 + \lambda^2 \gamma_n^2 } \le \frac{1}{2}  \right. \right\}  \\
\hspace{5pt}\text{ and }  \hspace{5pt} \mathcal{E}_1(\lambda) &=& \left\{ {x} \left|  |\widehat{\mathcal{G}}_n(\lambda \gamma_n) -  \overline{\mathcal{G}}_n(\lambda \gamma_n) | \le \lambda \gamma_n^2  \right. \right\}. 
\end{eqnarray*}
Controlling these sets allows us to quantify how close $\mathcal{G}_n(\delta)$ is to $\widehat{\mathcal{G}}_n(\delta)$ and how close $\widehat{\mathcal{G}}_n(\delta)$ is to $ \overline{\mathcal{G}}_n(\delta)$. For this purpose, we have the following lemma.   
\begin{lem} 
\label{KernelConcentration}
Assume that $K$ is a bounded kernel, so that point evaluations are bounded.  Then with $C$ representing possibly different constants at each occurrence, we have that
\[ \P\left\{ \mathcal{E}_0(\lambda) \right\}\ge  1 - C \exp ( - C n \gamma_n^2 \lambda^2  )  \hspace{5pt}, \quad \hspace{5pt} \P \left\{ \mathcal{E}_1(\lambda) \right\} \ge  1 - C \exp ( - C n \gamma_n^2 \lambda  ).  \]
\end{lem}

Note that on $\mathcal{E}_0(\lambda) \cup \mathcal{E}_1(\lambda)$, inclusion in  $\mathcal{E}_0(\lambda)$ guarantees that $\|f\|_n \le 2\|f\|_2 + \lambda \gamma_n$  and  $ \|f\|_2 \le 2\|f\|_n +  \lambda \gamma_n . $
% \[ \|f\|_n \le 2\|f\|_2 + \lambda \xi_n  \hspace{5pt} \text{ and } \hspace{5pt}   \|f\|_2 \le 2\|f\|_n +  \lambda \xi_n . \]
This in turn ensures that
\[  \mathcal{G}_n( \delta ) \le  \widehat{\mathcal{G}}_n ( 2 \delta + \lambda \gamma_n ) \hspace{5pt} \text{ and } \hspace{5pt}  \widehat{\mathcal{G}}_n ( \delta )  \le \mathcal{G}_n (  2\delta +  \lambda \gamma_n ) .  \]
In particular, taking $\delta = \lambda \gamma_n$ gives $\mathcal{G}_n( \lambda \gamma_n ) \le  \widehat{\mathcal{G}}_n ( 3 \lambda \gamma_n )$.   At the same time, inclusion in $\mathcal{E}_1(\lambda)$ guarantees that
\[ \widehat{\mathcal{G}}_n(\lambda \gamma_n) \le \overline{\mathcal{G}}_n(\lambda \gamma_n)  +  \lambda \gamma_n^2,   \] 
while the fact that $\gamma_n$ dominates the smallest positive solution of $\overline{\mathcal{G}}_n(\delta) \le \delta^2$ together with the non-increasing property of $\overline{\mathcal{G}}_n(\delta)/\delta$ guarantee that for $\lambda \ge 1$, 
$ \gamma_n \ge \overline{\mathcal{G}}_n( \gamma_n) / \gamma_n  \ge  \overline{\mathcal{G}}_n( \lambda \gamma_n)  /  \lambda \gamma_n  $
%\[ \gamma_n \ge  \frac{ \overline{\mathcal{G}}_n( \gamma_n)  }{ \gamma_n } \ge \frac{ \overline{\mathcal{G}}_n( \lambda \gamma_n)  }{ \lambda \gamma_n } \]
and so $ \overline{\mathcal{G}}_n( \lambda \gamma_n) \le \lambda \gamma_n^2$.  Piecing things together, this guarantees that on $\mathcal{E}_0(\lambda) \cup \mathcal{E}_1(\lambda)$ we have that $ 6 \lambda \gamma_n^2=3 \lambda \gamma_n^2 +  3\lambda \gamma_n^2\ge \overline{\mathcal{G}}_n( 3 \lambda \gamma_n ) +  3\lambda \gamma_n^2  \ge \widehat{\mathcal{G}}_n (3 \lambda \gamma_n) \ge \mathcal{G} (\lambda \gamma_n) $
%\bea
%6 \lambda \gamma_n^2 &=& 3 \lambda \gamma_n^2 +  3\lambda \gamma_n^2  \\
%&\ge&  \overline{\mathcal{G}}_n( 3 \lambda \gamma_n ) +  3\lambda \gamma_n^2  \\
%&\ge& \widehat{\mathcal{G}}_n (3 \lambda \gamma_n)  \\
%&\ge& \mathcal{G} (\lambda \gamma_n)
%\eea
Thus for $\lambda \ge 6$, on $\mathcal{E}_0(\lambda) \cup \mathcal{E}_1(\lambda)$ we have that $\mathcal{G}_n (\lambda \gamma_n) \le 6 \lambda \gamma_n^2 \le (\lambda \gamma_n)^2$ so that, with $\delta_n$ being the smallest positive solution of $\mathcal{G}_n (\delta) \le \delta^2$ we have $\delta_n \le \lambda \gamma_n$ and hence $\mathcal{E}_0(\lambda) \cup \mathcal{E}_1(\lambda) \subseteq  \{ \delta_n \le   \lambda \gamma_n \}.$
% \[   \mathcal{E}_0(\lambda) \cup \mathcal{E}_1(\lambda) \subseteq  \{ \delta_n \le   \lambda \gamma_n \}.   \] 
In particular, applying Lemma \ref{KernelConcentration}, we have that for $\lambda \ge 6$,
\bea
\P( \delta_n  >   \lambda \gamma_n ) &\le& \P \{  \mathcal{E}^c_0(\lambda) \} + \P \{ \mathcal{E}^c_1(\lambda) \}   \\
&\le& C \exp ( - C n \gamma_n^2 \lambda^2  ) + C \exp ( - C n \gamma_n^2 \lambda  ) \\
&\le& C \exp ( - C n \gamma_n^2 \lambda  ).
\eea
This provides the concentration of measure result we need to bound the expectation $\E_{\mathbb{X}_n} (\delta_n^2)$ and cap the rate of the kernel estimation method.  Thus we prove the following lemma. 
\begin{lem}
\label{lemma14}
Let $K$ be a bounded kernel, $\delta_n$ be the smallest positive solution to the empirical Gaussian complexity $2 \mathcal{G}_n(\delta) \le \delta^2$ and $\gamma_n$ the smallest solution to $2 \overline{\mathcal{G}}^K_n(\delta) \le \delta^2$.  Then, for $C$ possibly changing at each occurrence and $\lambda \ge 6$, we have that 
\[  \P( \delta_n  >   \lambda \gamma_n ) \le C \exp ( - C n \gamma_n^2 \lambda  ).\]
Consequently, provided $n \gamma_n^2 = O(1)$, it follows that $\E_{\mathbb{X}_n} \left( \delta_n^2 \right) \le C \gamma_n^2$.   
% {\color{red} The assertion that $\E_{\mathbb{X}_n} \delta_n^2 \le C \gamma_n^2$ does rely on something to the effect of $n \gamma_n^2 = O(1)$, say.  Make this clear.}
\end{lem}

Combining what has been shown gives the following theorem characterizing rate for convergence of reproducing kernel Hilbert space based methods in the case of random design.
\begin{thm}
Suppose that $K$ is a bounded kernel and $\gamma_n$ is the smallest solution to $2 \overline{\mathcal{G}}^K_n(\delta) \le \delta^2$.  Then it follows from the work done above that the reproducing kernel Hilbert space based procedure outlined in the intro to this section satisfies the oracle inequality
\[ \E \left( \| f - \hat{f}_{n, K}  \|_n^2 \right) \le C \left(  \gamma_n^2  + \E_{\mathbb{X}_n} \inf_{ \| g \|_{ \mathcal{H}} \le R }  \|f - g\|_n^2 \right). \]
% n^{- \alpha / (\alpha + 1)}  \right\}  .  \]
Consequently, if $f \in  \mathcal{H}$ and satisfies $\| f \|_{ \mathcal{H}} \le R$ we have $\E ( \| f - \hat{f}_{n, K}  \|_n^2) \le C \gamma_n^2.$  If the eigenvalues of $K$ decay like $\lambda_i \sim i^{-\alpha}$, this means that $\E( \| f - \hat{f}_{n, K}  \|_n^2) \le C n^{-\alpha/(\alpha + 1)}$.
%Give Sacks-Ylvisacker and radial basis function example.  
\end{thm}

\section*{Supplement to ``Nonparametric principal subspace regression''}
\label{SM}
This supplementary file includes additional simulation results, proofs of lemmas and convergence rates for the examples in Section \ref{Theoretical guarantees}. 

\section*{Acknowledgments}
\label{ACK}
Dengdeng Yu's research is partially supported by the Canadian Statistical Sciences Institute postdoctoral fellowship. Dehan Kong's research is partially supported by Natural Science and Engineering Research Council of Canada.
Fang Yao's research is partially supported by National
Natural Science Foundation of China with a Key grant 11931001 and a General grant 11871080, a Discipline Construction Fund at Peking University and Key Laboratory of Mathematical Economics and
Quantitative Finance (Peking University), Ministry of Education.

%\bibliographystyle{apalike}
%\bibliography{bibliography}

\thispagestyle{empty}

\newpage

\begin{center}
	%	\textit{Original Article}\bigskip
	
	{\Large {Supplement to ``Nonparametric principal subspace regression''}
	}
	
	\bigskip

	{\large Mark Koudstaal, Dengdeng Yu, Dehan Kong and Fang Yao} $\ $
	
\end{center}

The supplementary file contains additional simulation results, proofs for the auxiliary lemmas for main theorems and examples of the paper. Section \ref{additionalsim} contains additional simulation results where the error $ z_i $'s are assumed to be correlated across $ 1\leq i\leq n $ and the components within each $ z_i $ are also correlated. Section \ref{notation} contains notations used in the supplementary file. Finally, Section \ref{prooflemmas} contains the proofs of the auxiliary lemmas related to the main theorems and examples in the paper.

\section{Additional Simulation Results}\label{additionalsim}
In this section, we perform additional simulation studies, where the error $ z_i $'s are assumed to be correlated across $ 1\leq i\leq n $ and the components within each $ z_i $ are also correlated. Let $Z = (z_1 \ldots z_n)$ denote a $p\times n $ random error matrix with entries $z_{ji} $, $1 \leq j \leq p$ and $1 \leq i \leq n$, and $\mathrm{vec}$ denote the vectorization operator that stacks the columns of a matrix into a vector. We set $ \mathrm{vec}(Z)\sim \mathcal{N}(0, \Sigma) $, where $\Sigma= \Sigma_1\otimes \Sigma_2\in \mathbb{R}^{pn\times pn} $. Here $ \Sigma_1$ is a $ n\times n $ matrix representing the correlation within different subjects $ 1\leq i\leq n$, $ \Sigma_2 $ is a $ p\times p $ matrix representing the correlation among different components of $ z_i $, and $\otimes$ is the Kronecker product. This decomposition of $\Sigma$ is often named as the separability of the covariance matrix, which was studied in various literatures such as \citet{DeMunck2002,Dawid1981}. For $\Sigma_1$ and $ \Sigma_2 $, we assume they have autoregressive structures. In particular, we set the $(i_1,i_2)$-th element of $\Sigma_1 $ as $ 0.5^{|i_1-i_2|}$ for $1\leq i_1,i_2\leq n $ and the $(j_1,j_2)$-th element of $\Sigma_2 $ as $ 0.5^{|j_1-j_2|}$ for $ 1\leq j_1, j_2 \leq p $. For the other settings, they are the same as the ones in Section \ref{simulation} of the main paper. We still compare with curve-by-curve nonparametric recovery and consider same combinations of $ (n,p,q) $'s used in the main paper. The results are summarized in Table \ref{tab1:simu2}. 
{ \begin{table}[htbp]
\centering
\caption{Additional simulation results: the average estimation errors for our nonparametric principal subspace regression method (``NPSR error'') and the curve-by-curve nonparametric regression (``Nonparametric error''), and their associated standard errors in the parentheses are reported. The selected $q^*$ by $\AIC$ is also reported. The results are based on 100 Monte Carlo repetitions. }
\newsavebox{\tableboxb}
\begin{lrbox}{\tableboxb}
\begin{tabular}[l]{cccccccccccc}
%
%\hline \hline
$n$ &$q$& $p$& NPSR error  &  Nonparametric error  & $q^*$ \\
%\hline\\
\\
&  & 10 & 1.362 (0.062)  & 1.963 (0.050)  & 2.840 (0.099)   \\  
     & 2 & 20 & 2.052 (0.067)  & 3.250 (0.058)  & 2.770 (0.097) \\  
     &  & 40 & 3.123 (0.084)  & 5.657 (0.077)  & 2.520 (0.095)   \\  
                128&\\
%     & 2 & 60 & 4.294 (0.107)  & 7.894 (0.111)  & 2.370 (0.073)   \\  
%     & 2 & 80 & 5.235 (0.118)  & 9.652 (0.136)  & 2.320 (0.085)  \\  
     &  & 10 & 2.282 (0.097)  & 2.376 (0.082)  & 4.400 (0.102)   \\  
     & 4 & 20 & 3.553 (0.101)  & 4.028 (0.072)  & 4.150 (0.095)  \\  
     &  & 40 & 5.416 (0.095)  & 7.195 (0.104)  & 3.830 (0.079)  \\  
     \\
 %    & 4 & 60 & 7.207 (0.102)  & 9.720 (0.117)  & 3.690 (0.076)  \\  
 %    & 4 & 80 & 8.620 (0.124)  & 12.149 (0.137)  & 3.590 (0.079) \\  
  %   & 2 & 10 & 0.859 (0.047)  & 1.186 (0.036)  & 3.150 (0.117)  \\  
     &  & 20 & 1.170 (0.047)  & 2.002 (0.045)  & 2.730 (0.116)  \\  
     & 2 & 40 & 1.838 (0.073)  & 3.349 (0.042)  & 2.330 (0.074) \\  
     &  & 60 & 2.378 (0.059)  & 4.540 (0.059)  & 2.260 (0.058)  \\  
                256&\\
 %    & 2 & 80 & 2.788 (0.049)  & 5.659 (0.066)  & 2.130 (0.052)  \\  
 %    & 4 & 10 & 1.342 (0.056)  & 1.431 (0.058)  & 5.040 (0.104)  \\  
     &  & 20 & 2.121 (0.079)  & 2.407 (0.055)  & 4.760 (0.152)  \\  
     & 4 & 40 & 3.349 (0.074)  & 4.129 (0.059)  & 3.960 (0.082)  \\  
     &  & 60 & 4.344 (0.064)  & 5.590 (0.055)  & 3.700 (0.078)  \\  
     \\
 %    & 4 & 80 & 5.196 (0.073)  & 7.066 (0.077)  & 3.690 (0.081)   \\  
  %   & 2 & 10 & 0.540 (0.056)  & 0.683 (0.022)  & 3.970 (0.120)   \\  
  %   & 2 & 20 & 0.718 (0.031)  & 1.144 (0.030)  & 3.420 (0.173)   \\  
     &  & 40 & 1.118 (0.040)  & 1.975 (0.035)  & 2.460 (0.113)  \\  
     & 2 & 60 & 1.362 (0.046)  & 2.658 (0.038)  & 2.190 (0.061)  \\  
     &  & 80 & 1.635 (0.036)  & 3.306 (0.041)  & 2.080 (0.042)   \\  
                512&\\
 %    & 4 & 10 & 0.973 (0.049)  & 1.020 (0.047)  & 5.880 (0.116)   \\  
  %   & 4 & 20 & 1.412 (0.043)  & 1.542 (0.043)  & 6.480 (0.188)  \\  
     &  & 40 & 2.254 (0.055)  & 2.494 (0.043)  & 4.570 (0.138) \\  
     & 4 & 60 & 2.843 (0.067)  & 3.383 (0.044)  & 3.980 (0.081)  \\  
     &  & 80 & 3.299 (0.054)  & 4.270 (0.052)  & 3.750 (0.074)  \\  
\\
%    \hline \hline
\end{tabular}
\end{lrbox}
\label{tab1:simu2} 
\scalebox{1}{\usebox{\tableboxb}}
\end{table} }
From Table \ref{tab1:simu2}, one can see that our method still outperforms the curve-by-curve nonparametric regression for all cases. %For fixed $n$ and $q$, the recovery results from nonparametric principal subspace regression tend to improve at a faster rate as $p$ increases. Besides, one can see that the radial basis function \eqref{aicstar} is capable of choosing $ q^* $ close to the true value $ q $. 

\section{Notation}\label{notation}
We first introduce the notations used in the supplementary file, some of the notations have been introduced in the main paper. Recall we have proposed a nonparametric model $ y_i  = \sum_{k = 1}^q  f_k(x_i) u_k  + z_i \in \Rp $, where $ z_i=(z_{i1}, \ldots, z_{ip})^{\T} $ follows independent and identically distributed $ \mathcal{N}(0,\sigma ^2I_p) $. Without loss of generality, in the proof we assume $ \sigma^2=1 $. 
%{\color{blue} (Mark, double check whether we can assume $ \sigma^2=1 $.)}

Let $Y = (y_1 \ldots y_n) \in \mathbb{R}^{p\times n} $ be the response data matrix, $ \tilde{F}=(F(x_1) \ldots F(x_n)) $ $\in \mathbb{R}^{p\times n} $ and $ Z=(z_1 \ldots z_n) \in \mathbb{R}^{p\times n}  $, one can write $Y \equiv \tilde{F} + Z$. Let $ U=(u_1 \ldots u_q)\in \mathbb{R}^{p\times q} $, and $\tU=(\hat{u}_1 \ldots \hat{u}_q)\in \mathbb{R}^{p\times q}$ the estimate of $ U $. Define $ \tilde{f}_k=(f_k(x_1),\ldots, f_k(x_n))^{\T} \in \mathbb{R}^{n} $ for $ 1\leq k\leq q $ and $ \tilde{f}=(\tilde{f}_1 \ldots \tilde{f}_q)^{\T} \in \mathbb{R}^{q \times n} $ so that $\tilde F = U \tilde f$. We further define $ \tilde{\hat{f}}_k=(\hat{f}_k(x_1),\ldots, \hat{f}_k(x_n))^{\T}\in \mathbb{R}^{n} $ for $ 1\leq k\leq q $ and define $ \tilde{\hat{f}}=(\tilde{\hat{f}}_1 \ldots \tilde{\hat{f}}_q)^{\T} \in \mathbb{R}^{q \times n} $ so that we may write $\tilde{ \hat F} \equiv (\hat F(x_1) \ldots \hat F(x_n))  \in \mathbb{R}^{p\times n} $ as $\tilde{ \hat F} = \hat U \tilde{\hat{f}}$.

\section{Proofs of Auxiliary Lemmas and Theorems for Examples}\label{prooflemmas}

\begin{proof}[Proof of Lemma \ref{SingularNormMoment}]
%{\em Proof of Lemma \ref{SingularNormMoment}.} \ \
First notice that $\| Y \| \le \| \tilde F \| + \| Z \|$ implies $\|Y\|^4 \le C (  \| \tilde F
\|^4 + \| Z \|^4 )$.  As $$(\tilde f^\T \tilde f)_{ij} = \sum_{k=1}^q f_k(x_i) f_k(x_j) \hspace{5pt} \Rightarrow \hspace{5pt}
\Tr (\tilde f^\T \tilde f) = \sum_{i=1}^n \sum_{k=1}^q f_k^2 (x_i) = n \sum_{k=1}^q \|f_k\|_n^2,$$ 
%{\color{blue} (According to notation, this should be $(\tilde{f} {\tilde{f}}^{\T})_{ij} $, change the notation accordingly.)} 
%{\color{red} (I had previously defined the matrix $f$, using a transpose, which corresponds to the definition of $\tilde f$ above and puts it in $\mathbb{R}^{q \times n}$ (corresponding to $\tilde f^{\T}$ as you had defined it).  The way I had defined it, we have $\tilde F = U \tilde f$, which is important to the presentation and resolves the dimensional discrepancies you were worried about in the main theorem.  Let's please leave the definition this way.  Also, with this definition this can stay as is, as $\tilde f^\T \tilde f$.)}
one has
\begin{eqnarray*}  \| \tilde F \| &=& \max_i \sigma_i(\tilde F) \le \| \tilde F \|_{F} =  \left[ \Tr \{ (U\tilde f)^\T U\tilde f \} \right]^{1/2} =
\left\{ \Tr (\tilde f^\T \tilde f) \right\}^{1/2}  \\
&\le& \left( q n \max_k \| f_k \|_n^2 \right)^{1/2}.  
\end{eqnarray*}

%In both fixed and random design cases, given the assumption that $\max_k \| f_k \|
%_{\infty} \le B$, we have that
Given the assumption that $\max_k \| f_k \|
_{\infty} \le B$, we have that
\[ \| f_k \|_n^2  =  \frac{1}{n} \sum_{i=1}^n f^2_k(x_i) \le nB^2 / n = B^2. \]
Thus $\max_k \| f_k \|_n^2 \le B^2$ and so $\E \| \tilde F \|^4 \le C (qn)^2$.
Now, if $Z \in \mathbb{R}^{p \times n}$ is composed of independent and identically distributed $\mathcal{N}(0,1)$ entries, then
it is well known \citep{vershynin2010introduction} that there is a constant $C$ so that
for all $t>0$,
\[  \P \{ \| Z \| > C(p^{1/2} + n^{1/2} + t) \}  \le 2 \exp ( -t^2 ) . \] 
As shown below in Lemma \ref{MomentBd}, this implies that for some constant $C$,  $\E \| Z \|^4 \le C
\max(p^2 , n^2)$ and combining what has been shown gives the desired bound on $\E \|Y\|
^4$. 
\end{proof}

\begin{proof}[Proof of Lemma \ref{MomentBd}]
Separating on the value of $a$ we find that 
\[  \E (X^4) = \E \{ X^4 \ind_{(X \le a)} \} + \E \{ X^4 \ind_{(X > a)}\} \le a^4 + \E \{ X^4 \ind_{(X > a)} \}.
\]
Now notice that
\be
 \E\{ X^4 \ind_{(X > a)} \}&=& \int_{\Omega} X^4 (\omega)\ind_{\{X(\omega) > a\}} d\P (\omega) \nonumber \\
 &=&
 \int_{\Omega} \left( a^4 + 4 \int_a^{X(\omega)}  s^3 ds \right) d\P (\omega)
 \nonumber \\
 &=& \int_{\Omega} \left( a^4 + 4 \int_a^{\infty}  s^3 \ind_{\{s < X(\omega)\}}  ds \right) d
 \P (\omega)  \nonumber\\
  &=& a^4 + 4 \int_a^{\infty} s^3 \P(X > s) ds \nonumber \\
 &=& a^4 + 4 b \int_0^{\infty} (a + b t)^3 \P(X > a + bt)  dt \nonumber \\
 &\le & a^4 + 8 b \max(a^3 , b^3) \int_0^{\infty} (1 + t)^3 \exp (-t^2) dt \nonumber \\
 & \le & C \cdot \max(a^4 , b^4),  \nonumber
\ee
which concludes the proof of the lemma. 
\end{proof}

\begin{proof}[Proof of Lemma \ref{rndDes}]
Noting that for fixed $i,j$ we have
\[    \langle  f_i , f_j \rangle_n  - \langle  f_i , f_j \rangle  =  \frac{1}{n} \sum_{l
= 1}^n \{ f_i(x_l)f_j(x_l)  -  \langle  f_i , f_j \rangle \}  \]
guarantees that $\langle  f_i , f_j \rangle_n  - \langle  f_i , f_j \rangle$ is
expressible as the sum of $n$ independent and identically distributed  mean 0 random variables, each bounded by $2B^2 /
n$.   Hoeffding then gives that
\[   \P(   | \langle  f_i , f_j \rangle_n  - \langle  f_i , f_j \rangle |  > 2 \delta
B^2  )  \le 2 \exp \left( - \frac{n \delta^2 }{2} \right).   \]
Symmetry of inner product guarantees that there are $q(q + 1)/2$ distinct sums $|
\langle  f_i , f_j \rangle_n  - \langle  f_i , f_j \rangle |$ as we vary $i,j $ over $1,
\dots,q$  and so the first inequality of the theorem follows from a union bound.    

\end{proof}

%\begin{proof}[Proof of Lemma \ref{fxDes}]
%This is a direct application of Lemma \ref{Chazelle}, on noting that the conditions of the Lemma imply that $V(f_i f_j) = \| f_i f_j \|_{L^1} \le \| f_j \|_{\infty} \| f_i  \|_{L^1}  + \| f_i \|_{\infty} \| f_j  \|_{L^1} \le 2B_1 B_2 $.  
%\end{proof}

\begin{proof}[Proof of Lemma \ref{SampledSpan}]
%{\color{blue} (change the notation $ \tilde{f} $ accordingly, you may search all places with $f f^\T $ or $f^\T f$.)} 
First consider the matrix $\tilde f \tilde f^\T $ which, by construction, has elements $(\tilde f \tilde f^\T )_{ij} = n \langle f_i, f_j \rangle_n$.   Hence $\tilde f\tilde f^\T $ is real and symmetric and thus has an eigendecomposition $\tilde f\tilde f^\T  = V D V^\T $, with the columns of $V \in \mathbb{R}^{q \times q}$ forming an orthonormal basis and $D$ diagonal with nonnegative elements.   If the elements of $D$ are strictly positive, and hence $\tilde f\tilde f^\T $ is invertible, then we find that for each $\alpha \in \mathbb{R}^q$ there is an $\beta = \beta(\alpha) \in \mathbb{R}^q$ so that $\alpha = \tilde f \tilde f^\T  \beta = \tilde f \gamma$, with $\gamma = \tilde f^\T  \beta \in \mathbb{R}^n$.   Thus there is at least one $\gamma \in  \mathbb{R}^n$ so that $U \alpha = U \tilde f \gamma = \tilde F \gamma$ and hence $\Sp ( U ) = \Sp( \tilde F) = \Sp ( P )$.    Thus to prove the claim of the theorem, it is enough to show that the elements of $D$ are positive under the assumptions of the theorem.  

% By the Gershgorin disk theorem and using the form of $\tilde f\tilde f^\T $, if $d$ is an eigenvalue of $\tilde f\tilde f^\T $, then
%\[ d / n  \ge \min_i \left\{  \langle f_i, f_i \rangle_n - \sum_{j \ne i} | \langle f_i, f_j \rangle_n | \right\}.  \]
%As the $f_k \in L^2$ are nonzero and orthogonal, there is a  $c > 0$ so that $c \le \min_k \| f_k \|^2_{L^2}$ and $\langle f_i , f_j \rangle_{L^2} = 0$ whenever $i \ne j$.  Under the fixed design case, Lemma \ref{fxDes} implies that the above bound for $d/n$ has a lower bound of
%\[ d / n  \ge  c - qC / n ,  \]
%which is positive when $n$ is large enough, i.e. $n > qC/c$.  Hence the event under consideration goes from probability $0$ to probability $1$ at some threshold and we may lower bound this with $1 - B / n^2$ for some $B \ge 0$.  In the random design case, consider the event $\mathcal{D}$ given by   
%\[    \mathcal{D}  = \left\{   \max_{i,j}  | \langle  f_i , f_j \rangle_n  - \langle  f_i , f_j
%\rangle |   \le  B \left( \frac{\log n}{n}  \right)^{1/2}    \right\}  ,   \]
%which Lemma \ref{rndDes} implies for properly chosen $B$ has probability (with possibly different constant $B$) $\ge 1 - B/n^2$.  Now, on this event the above bound for $d/n$ has a lower bound of
%\[ d / n  \ge  c - q B \left( \frac{\log n}{n}  \right)^{1/2}  ,  \]
%which is also positive for $n$ large enough, since $\log n = o(n)$.  This shows that in both cases, with the quoted probability, we eventually have that $\tilde f\tilde f^\T $ is invertible and hence may reach the conclusion of the theorem.  

 By the Gershgorin disk theorem and using the form of $\tilde f\tilde f^\T $, if $d$ is an eigenvalue of $\tilde f\tilde f^\T $, then
\[ d / n  \ge \min_i \left\{  \langle f_i, f_i \rangle_n - \sum_{j \ne i} | \langle f_i, f_j \rangle_n | \right\}.  \]
As the $f_k \in L^2$ are nonzero and orthogonal, there is a  $c > 0$ so that $c \le \min_k \| f_k \|^2_{L^2}$ and $\langle f_i , f_j \rangle_{L^2} = 0$ whenever $i \ne j$.  Consider the event $\mathcal{D}$ given by   
\[    \mathcal{D}  = \left\{   \max_{i,j}  | \langle  f_i , f_j \rangle_n  - \langle  f_i , f_j
\rangle |   \le  B \left( \frac{\log n}{n}  \right)^{1/2}    \right\}  ,   \]
which Lemma \ref{rndDes} implies for properly chosen $B$ has probability (with possibly different constant $B$) $\ge 1 - B/n^2$.  Now, on this event the above bound for $d/n$ has a lower bound of
\[ d / n  \ge  c - q B \left( \frac{\log n}{n}  \right)^{1/2}  ,  \]
which is positive for $n$ large enough, since $\log n = o(n)$.  This shows that, with the quoted probability, we eventually have that $\tilde f\tilde f^\T $ is invertible and hence may reach the conclusion of the theorem. 

\end{proof}

\begin{proof}[Proof of Lemma \ref{QuantitativeSubspaceAngles}]
The eigenvalues of $A = \tilde F \tilde F^\T $ are the same as the
squared singular values of $\tilde F$.  Further, we notice that we may write the matrix $\tilde f\tilde f^\T $ as
\[     \tilde f \tilde f^\T  =  n \sum_{i=1}^q \langle f_i, f_i \rangle_{L^2} e_i e_i^\T     +  \Delta  = \tilde f \tilde f^\T  =  n \sum_{i=1}^q \sigma_i^2 e_i e_i^\T     +  \Delta ,    \]
where the matrix $\Delta$ is composed of elements $\Delta_{ij} =n ( \langle f_i, f_i \rangle_{n} -  \langle f_i, f_i \rangle_{L^2} )$.  
Thus we have that
\[   A =  U \tilde f \tilde f^\T  U^\T   = n \sum_{i=1}^q \sigma_i^2 u_i u_i^\T    +  U \Delta U^\T  , \]
and $A$, $\Delta$, and thus $U \Delta U^\T $, are both real and symmetric.  Therefore Lemma \ref{WeylThm} implies that if $\mu$ is one of the $q$ largest eigenvalues of $A$, then
\[    \min_{i \le q} | \mu  - n \sigma_i^2 | \le  \| U \Delta U^\T  \|_2  .  \]
As the nullspace of $U \Delta U^\T $ consists of all vectors orthogonal to $\Sp(U)$,
in bounding
\[    \| U \Delta U^\T  \|_2  =  \sup_{x \ne 0 , \| x \| \le 1}   \|  U \Delta U^\T  x   \|_2
,    \]
we may restrict to considering $x \in \Sp(U)$.  Thus we may write $x = U \alpha$ for
some $\alpha$ satisfying $ 0 < \| \alpha \| \le 1$ and note that for such $x$,
\[    \|  U \Delta U^\T  x   \|_2^2    =  \|  \Delta \alpha \|_2^2 \le \sum_{i,j = 1}^q
\Delta_{ij}^2 \le (q B)^2 , \]

%where $B$ is any bound for $\Delta_{ij}$ satisfying $\max_{ij} |\Delta_{ij}| \le B$.  Thus we have that $ \| U \Delta U^\T  \|_2  \le qB$ and using the bounds of Lemmas \ref{rndDes} and \ref{fxDes} with the probability given there, we have that in both the fixed and random design cases, for appropriate $C$, if $\mu$ is one of the $q$ largest eigenvalues of $A$, then
%\[    \min_{i \le q} | \mu  - n \sigma_i^2 | \le C q \left(  n \log n \right)^{1/2} , \]
%with the quoted probability.  As noted at the outset, the eigenvalues of $A$ are squared singular values of $\tilde F$ and so this implies that if $\gamma$ is one of the top $q$ singular values of $\tilde F$ then
%\[    \min_{i \le q} | \gamma - n^{1/2} \sigma_i | \le Cq^{1/2} \left( n \log n \right)^{1/4}    \]
where $B$ is any bound for $\Delta_{ij}$ satisfying $\max_{ij} |\Delta_{ij}| \le B$.  Thus we have that $ \| U \Delta U^\T  \|_2  \le qB$ and using the bound of Lemma \ref{rndDes} with the probability given there, we have that, for appropriate $C$, if $\mu$ is one of the $q$ largest eigenvalues of $A$, then
\[    \min_{i \le q} | \mu  - n \sigma_i^2 | \le C q \left(  n \log n \right)^{1/2} , \]
with the quoted probability.  As noted at the outset, the eigenvalues of $A$ are squared singular values of $\tilde F$ and so this implies that if $\gamma$ is one of the top $q$ singular values of $\tilde F$ then
\[    \min_{i \le q} | \gamma - n^{1/2} \sigma_i | \le Cq^{1/2} \left( n \log n \right)^{1/4}    \]
%with the quoted probability {\color{blue} (write down the detailed probability)}.  This concludes the proof of the theorem as it entails that each of the top $q$ singular values of $\tilde F$ looks like $(1 + o(1)) n^{1/2} \sigma_i$ for some $i = 1,\dots,q$.
with probability $\ge 1 - B / n^2$ for some $B > 0$.  This concludes the proof of the theorem as it entails that each of the top $q$ singular values of $\tilde F$ looks like $(1 + o(1)) n^{1/2} \sigma_i$ for some $i = 1,\dots,q$.
\end{proof}

%The following theorem ensures that, in a certain sense, reasonable $F:[0,1] \rightarrow \Rp$
%has a singular value decomposition type representation, which further supports the form of function proposed in the paper. {\color{blue} (I am confused about this theorem. From my understanding, the inner product here may not be the inner product defined in our main paper, right? If this is the case, then if we change the inner product to a different one compared to the main paper, we may need a completely different estimation procedure. If my understanding is correct, I am not sure whether we need to include this theorem.)}
%{\color{red} (This theorem more or less guarantees the form of function we assert in at the outset of the paper; all we have done is assume dimension reduction on this form; i.e. that $p-q$ of the $\sigma_i = 0$)}
%\begin{thm}  \label{svdRepresentation}
%Suppose that $F:[0,1]^d \rightarrow \mathbb{R}^p$, which may be written $F = (F_1,
%\dots,F_p)$, satisfies $F_k \in L^2[0,1]^d$ for $ 1\leq k\leq p $.  Then, under appropriate inner products, $F$
%has a singular value type decomposition
%\[   F = \sum_{k=1}^p \sigma_k  u_k \otimes v_k=\sum_{k=1}^p f_k \otimes u_k,    \]
%where $ f_k= \sigma_k v_k $, $v_k$ are orthonormal in $L^2[0,1]^d$, $u_k$ are orthonormal in $\Rp$ and $\sigma_k
%\ge 0$ are decreasing.
%\end{thm}

\begin{proof}[Proof of Theorem \ref{LPStheorem}]
Let $b(x) = \E_{f, \mathbb{X}} \hat{f}(x)  -  f(x)$ denote the bias, conditioned on design, of the estimator $ \hat{f}(x)$ at $x$.  Then we find that
\bea
b(x) &=& \sum_{i = 1}^n f(x_{(i)}) W_{n, i}(x)    -  f(x)  =   \sum_{i = 1}^n  \left\{ f(x_{(i)}) - f(x) \right\} W_{n, i}(x)    \nonumber \\
&=&  \sum_{i = 1}^n  \left\{ f(x_{(i)}) - f(\delta_i)  \right\} W_{n, i}(x)  +  \sum_{i = 1}^n  \left\{  f(\delta_i) -  f(x) \right\} W_{n, i}(x) 
\eea
and hence
\[ |b(x)| \le  \underbrace{ \left| \sum_{i = 1}^n  \left\{ f(x_{(i)}) - f(\delta_i)  \right\} W_{n, i}(x) \right| }_{\text{I}} +   \underbrace{ \left| \sum_{i = 1}^n  \left\{  f(\delta_i) -  f(x) \right\} W_{n, i}(x) \right| }_{\text{II}}.  \]
Now $\text{II}$ is a deterministic quantity and is bounded in \cite{Tsybakov:2008:INE:1522486} to the order of $h^{\beta}$ for the H\"older class $\Sigma(\beta, L)$.   So starting from the fact that for $x_i$ we have
\bea
\E_f \{ \hat{f}(x_i)    -  f(x_i) \}^2 &=& \E_f  \E_{f, \mathbb{X}}  \{ \hat{f}(x_i)    -  f(x_i) \}^2  \nonumber \\
&=& \E_f  \E_{f, \mathbb{X}}  \{ \hat{f}(x_i)    -  \E_{f, \mathbb{X}} \hat{f}(x_i)  +  \E_{f, \mathbb{X}} \hat{f}(x_i)  -  f(x_i) \}^2  \nonumber \\
&=& \E_f   \left[ \V_{ f, \mathbb{X}} \{  \hat{f}(x_i) \}  +  b^2 (x_i)  \right]   \nonumber
\eea
From above (using the results from \cite{Tsybakov:2008:INE:1522486}), we know that
\[  b^2(x) \le 2 ( \text{I}^2  + \text{II}^2)  \le 2 (\text{I}^2 + C h^{2 \beta})  \]
and that
\[  \V_{ f, \mathbb{X}} \{   \hat{f}(x) \}  \le  \frac{C}{nh}, \]
which together give that 
\[ \E_f \{ \hat{f}(x_i)    -  f(x_i) \}^2 \le C \left( h^{2\beta} + \frac{1}{nh} \right) +  2 \E_f \text{I}^2. \]
Now $f$ is differentiable and so $|f(x) - f(y)| \le C |x - y| $, so we find that
\[  \text{I}  \le   \sum_{i = 1}^n  \left| f(x_{(i)}) - f(\delta_i)  \right|  \left| W_{n, i}(x) \right|  \le  C \sum_{i = 1}^n  \left| x_{(i)} - \delta_i   \right|  \left| W_{n, i}(x) \right|  . \]
Applying Cauchy-Schwarz to the right hand side and using the properties of $W_{n, i}(x)$ from \cite{Tsybakov:2008:INE:1522486} gives
\[ \text{I}^2   \le  C \sum_{i = 1}^n  \left| x_{(i)} - \delta_i   \right|^2  \sum_{i = 1}^n \left| W_{n, i}(x) \right|^2 \le \frac{C}{nh} \sum_{i = 1}^n  \left| x_{(i)} - \delta_i   \right|^2. \]
Thus
\[ \E_f \text{I}^2   \le \frac{C}{nh} \sum_{i = 1}^n   \V\{ x_{(i)} \}  \le \frac{C}{nh} \sum_{i = 1}^n  \frac{i}{n^2}  \le \frac{C}{nh} . \]
This shows that
\[    \E_f \{ \hat{f}(x_i)    -  f(x_i) \}^2 \le C \left( h^{2\beta} + \frac{1}{nh} \right),  \]
and hence
\[ \E_f ( \| \hat{f} - f \|_n^2) \le C \left( h^{2\beta} + \frac{1}{nh} \right).  \]
Choosing $h \sim n^{-1/(2\beta + 1)}$ gives that $ \E_f ( \| \hat{f} - f \|_n^2)  \le C n^{- 2\beta / (1 + 2\beta)}$, which finishes the proof.
\end{proof}

\begin{proof}[Proof of Theorem \ref{OrrTES}]
Using the quantities defined above, notice that we may bound $( f - \hat{f}_{n, K}  )^2$ as
\bea
( f - \hat{f}_{n, K}  )^2  &=& \left\{ \sum_{k = 1}^K (c_k - \hat{c}_k) \varphi_k   +  f_K  \right\}^2    \nonumber \\
&\le& 2  \left\{ \sum_{k = 1}^K (c_k - \hat{c}_k) \varphi_k  \right\}^2  + 2 f_K^2 .  \nonumber 
\eea
Expanding the first term, this gives that
\bea
\| f - \hat{f}_{n, K}  \|_n^2  &\le& \frac{2}{n} \sum_{i=1}^n \left\{  \sum_{u, v = 1}^K (c_u - \hat{c}_u) (c_v - \hat{c}_v) \varphi_u(x_i) \varphi_v(x_i)   +  f_K^2(x_i)  \right\}.
\eea
Now notice that
\bea
\E_{f, \mathbb{X}} \{ (c_u - \hat{c}_u) (c_v - \hat{c}_v) \}  &=&   \E_{f, \mathbb{X}} \{ (c_u -  \overline{c}_u +  \overline{c}_u -  \hat{c}_u) (c_v -  \overline{c}_v +  \overline{c}_v - \hat{c}_v)\}   \nonumber \\
&=& (c_u -  \overline{c}_u)  (c_v -  \overline{c}_v ) + \E_{f, \mathbb{X}} \{ ( \overline{c}_u - \hat{c}_u) ( \overline{c}_v - \hat{c}_v)\} . \nonumber 
\eea
Then because
\bea
 \E_{f, \mathbb{X}}  \{ ( \overline{c}_u - \hat{c}_u) ( \overline{c}_v - \hat{c}_v)  \} &=& \E_{f, \mathbb{X}} \{ \frac{1}{n^2} \sum_{i, j = 1}^n z_i z_j \varphi_u(x_i) \varphi_v(X_j) \}    \nonumber  \\
 &=& \frac{1}{n^2} \sum_{i = 1}^n  \varphi_u(x_i) \varphi_v(x_i )  =  \langle  \varphi_u,  \varphi_v \rangle_n / n   \nonumber
\eea
we arrive at
\[ \E_{f, \mathbb{X}} \{ (c_u - \hat{c}_u) (c_v - \hat{c}_v) \} =  (c_u -  \overline{c}_u)  (c_v -  \overline{c}_v )  + \langle  \varphi_u,  \varphi_v \rangle_n / n,   \]
and this gives that
\begin{eqnarray*}
&& \E_{f, \mathbb{X}} \{  \| f - \hat{f}_{n, K}  \|_n^2\}  \\
&\le& 2 \sum_{u, v = 1}^K \left\{   (c_u -  \overline{c}_u)  (c_v -  \overline{c}_v )  + \frac{ \langle  \varphi_u,  \varphi_v \rangle_n }{ n } \right\}  \langle  \varphi_u,  \varphi_v \rangle_n
+  2 \| f_K \|_n^2.  \nonumber
\end{eqnarray*}
Expanding, we find that
\bea
\E_f  \{ \langle  \varphi_u,  \varphi_v \rangle_n^2 \}  &=&  \E_f \left\{  \frac{1}{n^2}  \sum_{i, j = 1}^n \varphi_u (x_i) \varphi_v (x_i) \varphi_u (X_j) \varphi_v (X_j) \right\} \nonumber \\
&=&  \frac{1}{n^2}  \sum_{i, j = 1}^n \left\{ (1 - \delta_{ij}) \langle \varphi_u, \varphi_v \rangle^2 + \delta_{ij} \langle \varphi_u^2, \varphi_v^2 \rangle \right\} \nonumber \\
&=& \frac{ (n - 1) \delta_{uv} + \langle \varphi_u^2, \varphi_v^2 \rangle }{n} \le  \frac{ (n - 1) \delta_{uv} +  1 }{n} 
\eea
Now we need to look at the other term.   First notice that we may write
\bea
 \E_f  \{ (c_u -  \overline{c}_u)  (c_v -  \overline{c}_v ) \langle  \varphi_u,  \varphi_v \rangle_n\}  &=& \frac{1}{n^3} \sum_{i, j, k = 1}^n A_{ijk} 
\eea
Where
\bea
 A_{ijk} &=&  \E_f \left[ \{ f(x_i) \varphi_u(x_i) - c_u\} \{ f(X_j) \varphi_v(X_j) - c_v \} \varphi_u(X_k) \varphi_v(X_k) \right]\nonumber \\
 &=& \delta_{ij} \left[ (1 - \delta_{kj} )  \delta_{uv}  \V \{ (f \varphi_v)(X) \}   +  \delta_{kj} D_{uv}  \right]
 \eea
and
\begin{eqnarray*}  D_{uv} &=&   \E_f \left[ \{ f(X) \varphi_u(X) - c_u\} \{f(X) \varphi_v(X) - c_v\} \varphi_u(X) \varphi_v(X) \right]  \\
 & & \hspace{200pt} \text{ with } \hspace{5pt} X \sim U[0, 1] .  
 \end{eqnarray*}
As $| \varphi_u | \le 1$, we have $\V \{ (f \varphi_v)(X) \}  \le \|f\|^2_{L^2}$.   Thus we get that 
\[   A_{ijk} \le \delta_{ij} \left\{ (1 - \delta_{kj} )  \delta_{uv}   \|f\|^2_{L^2}   +  \delta_{kj} D_{uv}  \right\}. \]
When $u = v$, one has
\begin{eqnarray*}
&& \E_f \{ f(X) \varphi_u(X) - c_u \} \{ f(X) \varphi_v(X) - c_v \} \varphi_u(X) \varphi_v(X)  \\
 &\le&  \V ( (f \varphi_v)(X) )  \le \|f\|^2_{L^2}.  
 \end{eqnarray*}
Similarly, when $u \ne v$ we may apply Cauchy-Schwarz inequality
\bea
&&\E_f \left[ (f(X) \varphi_u(X) - c_u) (f(X) \varphi_v(X) - c_v) \varphi_u(X) \varphi_v(X) \right]  \nonumber \\
 &\le& \left[   \V \{ (f \varphi_u)(X) \} \right]^{1/2}  \left[   \V \{ (f \varphi_v)(X) \} \right]^{1/2}  \le \|f\|_{L^2}^2 .   \nonumber
\eea
Piecing things together then gives that
\[  A_{ijk} \le  \delta_{ij} \left\{ (1 - \delta_{kj} )  \delta_{uv}   +  \delta_{kj}  \right\} \| f \|_{L^2}^2,    \]
and so
\[   \E_f  \{ (c_u -  \overline{c}_u)  (c_v -  \overline{c}_v ) \langle  \varphi_u,  \varphi_v \rangle_n  \} \le \frac{ (n(n - 1) \delta_{uv}  + n ) \|f\|_{L^2}^2 }{n^3} .  \]
Then putting everything together gives that 
\bea
\E_f \left[ \sum_{u, v = 1}^K \left\{   (c_u -  \overline{c}_u)  (c_v -  \overline{c}_v )  + \frac{ \langle  \varphi_u,  \varphi_v \rangle_n }{ n } \right\}  \langle  \varphi_u,  \varphi_v \rangle_n  \right]  \nonumber \\
\le  \frac{ n^2 K + n K^2  }{  n^3  }   \|f\|_{L^2}^2  +  \frac{ nK  + K^2 }{ n^2 }  \le \frac{ (1 + K / n ) (1 + \|f\|_{L^2}^2 )  K }{ n } .
\eea
This, in turn, gives that
\[ \E_f   ( \| f - \hat{f}_{n, K}  \|_n^2  ) \le 2 \frac{ (1 + K / n ) (1 + \|f\|_{L^2}^2 )  K }{ n }  +  2 \E_f (\| f_K \|_n^2).    \]
\end{proof}

\begin{proof}[Proof of Lemma \ref{KernelConcentration}]
The first inequality follows from the fact that $\overline{\mathcal{R}}_n(\delta) \le \overline{\mathcal{G}}_n^K(\delta)$ together with an application of theorem 14.1 from \cite{Wainwright2019}.  The second will follow from the fact hat $\widehat{\mathcal{G}}_n(\delta)$ is self-bounding \cite{Boucheron2009}, together with corresponding concentration results for this class of functions.   To see this, let $Z = h(x_1, \dots, x_n) = n \widehat{\mathcal{G}}_n(\delta)$ and 
\[ Z_i = h(x_1, \dots, x_{i-1}, x_{i+1}, \dots, x_n) = \E_{g} \left\{ \sup_{f \in \partial \mathcal{H}(\delta)}   \sum_{j \ne i} g_j f(x_j) \right\}.\]
%\[ Z = h(x_1, \dots, x_n) = n \widehat{\mathcal{G}}_n(\delta)  \hspace{5pt} \text{ and } \hspace{5pt}  Z_i = h(x_1, \dots, x_{i-1}, x_{i+1}, \dots, x_n) = \E_{g} \sup_{f \in \partial \mathcal{H}(\delta)}   \sum_{j \ne i} g_j f(x_j) .\]
Suppose further that for each $g$ the $\sup$ is attained at  $f^* = f^*(g) \in  \partial \mathcal{H}(\delta)$ (otherwise use limits and Fatou's lemma) so that we have
\bea
Z - Z_i &=& \E_g \left\{  n \sup_{f \in \partial \mathcal{H}(\delta)}  \langle g , f \rangle_n -  \sup_{f \in \partial \mathcal{H}(\delta)}  \sum_{j \ne i} g_j f(x_j) \right\} \\
&=&  \E_g \left\{   n \langle g , f^*(g)  \rangle_n -  \sup_{f \in \partial \mathcal{H}(\delta)}  \sum_{j \ne i} g_j f(x_j) \right\} \\
&\le&  \E_g \left\{  n \langle g , f^*(g)  \rangle_n -  \sum_{j \ne i} g_j f^*(g)(x_j) \right\} \\
&=&  \E_g \left\{   g_i  f^*(g)(x_i)  \right\}  \le  C,
\eea
since we have assumed the kernel is bounded, and thus point evaluations are bounded for all $f \in \partial \mathcal{H}(\delta)$.
Thus, by re-scaling we may assume that $0 \le Z - Z_i \le 1$.    Similarly, adding the upper bound gives that
\[ \sum_{i=1}^n (Z - Z_i) \le n \widehat{\mathcal{G}}_n(\delta)  = Z, \]
and this has shown that $Z = n \widehat{\mathcal{G}}_n(\delta) / C$ is self-bounding for some C, and so by the results of  \cite{Boucheron2009} this immediately implies that for all $u > 0$
\[ \P \left\{  |n \widehat{\mathcal{G}}_n(\delta) - n \overline{\mathcal{G}}_n(\delta) | > C u  \right\} \le 2 \exp \left[ - \frac{u^2}{2 \{ n \overline{\mathcal{G}}_n(\delta)/C + u  \}} \right] , \]
which on making the transformation $u \rightarrow nu/C$ gives that for $u>0$
\[ \P \left\{  |\widehat{\mathcal{G}}_n(\delta) -  \overline{\mathcal{G}}_n(\delta) | >  u  \right\} \le 2 \exp \left[ - \frac{ n u^2 / C }{2 \{ \overline{\mathcal{G}}_n(\delta) + u  \}} \right] . \]
Now let $\gamma_n$ be the smallest positive solution to $\overline{\mathcal{G}}^K_n(\delta) = \delta^2$ and take $u = \lambda \gamma_n^2$ to see that with $C$ possibly different,
\[ \P \left\{ |\widehat{\mathcal{G}}_n(\lambda \gamma_n) -  \overline{\mathcal{G}}_n(\lambda \gamma_n) | > \lambda \gamma_n^2  \right\} \le 2 \exp \left\{ -  C \frac{ n \gamma_n^4 \lambda^2}{ \overline{\mathcal{G}}_n(\lambda \gamma_n) + \lambda \gamma_n^2  } \right\} . \]
Using the fact that the function $ \overline{\mathcal{G}}_n(\delta)  / \delta $ is non-increasing then gives that for $\lambda \ge 1$, 
\[  \overline{\mathcal{G}}_n(\lambda \gamma_n)  / (\lambda \gamma_n)  \le \overline{\mathcal{G}}_n( \gamma_n)  /  \gamma_n  \le  \gamma_n, \] 
so that $ \overline{\mathcal{G}}_n(\lambda \gamma_n) \le \lambda \gamma_n^2$. Substituting this bound shows that for $\lambda \ge 1$, with $C$ representing possibly different constants
\[ \P \left\{ \mathcal{E}^C_1(\lambda) \right\} = \P \left\{  |\widehat{\mathcal{G}}_n(\lambda \gamma_n) -  \overline{\mathcal{G}}_n(\lambda \gamma_n) | > \lambda \gamma_n^2  \right\} \le  C \exp \left(  - C n \gamma_n^2 \lambda   \right). \] 
%Another, potentially more useful way to frame (or use) this inequality is to notice that since $\overline{\mathcal{G}}_n(\lambda \gamma_n) \le \lambda \gamma_n^2$, we find that
%\bea
%\P \left(  \widehat{\mathcal{G}}_n(\lambda \overline{\delta}_n)  >  2 \lambda \overline{\delta}_n^2  \right) &\le& \P \left(  \widehat{\mathcal{G}}_n(\lambda \overline{\delta}_n)  >  \overline{\mathcal{G}}_n(\lambda \overline{\delta}_n) +  \lambda \overline{\delta}_n^2  \right) \\
%&\le& \P \left(  |\widehat{\mathcal{G}}_n(\lambda \overline{\delta}_n) -  \overline{\mathcal{G}}_n(\lambda \overline{\delta}_n) | > \lambda \overline{\delta}_n^2  \right)  \\
%&\le& 2 \exp \left(  - \frac{ n \overline{\delta}_n^2 \lambda }{4 }   \right).
%\eea
%In fact, we can remove the second inequality there and use the self-bounding inequalities for just the upper tail to improve this to a single factor of the exponential bound, so that we remove the 2.  
\end{proof}

\begin{proof}[Proof of Lemma \ref{lemma14}]
The concentration inequality has been established in the discourse above, so we shall focus on proving the assertion that $\E_{\mathbb{X}_n} \left( \delta_n^2 \right) \le C \gamma_n^2$.  To see this, notice that the concentration result above guarantees that
\bea
\E_{\mathbb{X}_n} \left( \delta_n^2 \right)&=& \int_0^{\infty} 2s \P(\delta_n > s) ds \\
&=& 2 \left\{   \int_0^{6\gamma_n} s \P(\delta_n > s) ds +  \int_{6 \gamma_n}^{\infty} s \P(\delta_n > s) ds  \right\} \\
&\le& 2 \left\{   \int_0^{6\gamma_n} s ds +  \int_{6 }^{\infty} \lambda \gamma_n \P(\delta_n > \lambda \gamma_n) \gamma_n d\lambda  \right\}  \\
&\le&  C \gamma_n^2 \left\{  1 +  \int_{6 }^{\infty} \lambda  \P(\delta_n > \lambda \gamma_n) d\lambda  \right\} \\
&\le& C \gamma_n^2 \left\{  1 +  \int_{6 }^{\infty} \lambda \exp ( - C n \gamma_n^2 \lambda  )  d\lambda  \right\} \le C \gamma_n^2.
\eea
\end{proof}

\bibliographystyle{apalike}
\bibliography{bibliography}

\end{document}